\documentclass{amsart}
\usepackage{graphicx}
\usepackage{amsmath}
\usepackage{amssymb}
\usepackage[all,cmtip]{xy}
\usepackage{epsfig}
\usepackage[normal]{subfigure}
\usepackage{latexsym} 
\usepackage{color}
\usepackage{url}

\newtheorem{theorem}{Theorem}[section]
\newtheorem{lemma}[theorem]{Lemma}
\newtheorem{corollary}[theorem]{Corollary}

\newtheorem{proposition}[theorem]{Proposition}
\newtheorem{question}[theorem]{Question}
\newtheorem{problem}[theorem]{Problem}

\newtheorem{definition}[theorem]{Definition}

\newtheorem{remark}[theorem]{Remark}
\newtheorem{observation}[theorem]{Observation}

\newcommand\Z{\mathbb{Z}}
\newcommand\KI{\mathbb{K}_\infty}
\newcommand\Kn{\mathbb{K}_n}
\newcommand\Kfour{\mathbb{K}_4}

\newcommand\thegraph{\ensuremath{\mathcal{B}}}
\newcommand\thehypgraph{\ensuremath{\mathcal{B}_H}}

\title{The Big Dehn Surgery Graph and the link of $S^3$} 
\author{Neil R. Hoffman and Genevieve S. Walsh} 
\subjclass[2000]{Primary 57M25; Secondary 57M50}

\begin{document} 
%\linenumbers
\dedicatory{Dedicated to Bill Thurston}
\maketitle

\begin{abstract} 
In a talk at the Cornell Topology Festival in 2004, W. Thurston discussed a graph which we call ``The Big Dehn Surgery Graph", \thegraph.  Here we explore this graph, particularly the link of $S^3$, and prove facts about the geometry and topology of \thegraph.  We also investigate some interesting subgraphs and pose what we believe are important questions about \thegraph. \end{abstract} 

\section{Introduction} \label{sect:Intro}
In unpublished work, W. Thurston described a graph that had a vertex $v_M$ for each closed, orientable, 3-manifold $M$ and an edge between two distinct vertices $v_M$ and $v_{M'}$,  if there exists a Dehn surgery between $M$ and $M'$.  That is, there is a knot $K \subset M$ and $M'$ is obtained by non-trivial Dehn surgery along $K$ in $M$.  The edges are unoriented since $M$ is also obtained from $M'$ via Dehn surgery.  Roughly following W. Thurston, we will call this graph the Big Dehn Surgery Graph, denoted by $\thegraph$. We will sometimes denote the vertex $v_M$ by $M$.   If $M$ and $M'$ are obtained from one another via Dehn surgery along two distinct knots, we do not make two distinct edges, although this would also make an interesting graph. We first record some basic properties of $\thegraph$.  These follow from just some of the extensive work that has been done in the field of Dehn surgery. 

\begin{proposition} \label{prop:basic} 
The graph $\thegraph$ has the following basic properties: 
 (i) \thegraph \  is connected; 
(ii) \thegraph \  has infinite valence; 
and (iii) \thegraph \  has infinite diameter. 
\end{proposition} 

 The graph $\thegraph$ is connected by the beautiful work of Lickorish \cite{LICKORISH} and Wallace \cite{WALLACE} who independently showed that all closed, orientable $3-$manifolds can be obtained by surgery along a link in $S^3$.  That every vertex $v_M$ in $\thegraph$ has infinite valence can be seen, amongst other ways, by constructing a hyperbolic knot $K$  in $M$ via the work of Myers in \cite{MYERSEXCELLENT}.  Then by work of Thurston \cite{THURSTONNOTES} all but finitely many fillings are hyperbolic, and the volumes of the filled manifolds approach the volume of the cusped manifold.  The graph \thegraph \ has infinite diameter since the rank of $H_1(M, \mathbb{R})$ can change by at most one via drilling and filling, and there are 3-manifolds with arbitrarily high rank. 
 
  The Lickorish proof explicitly constructs a link, and therefore allows us to describe a natural notion of distance. A shortest path from $v_{S^3}$ to $v_{M}$ in $\thegraph$ counts the minimum number of components needed for a link in $S^3$ to admit $M$ as a surgery.  We will refer to the number of edges in a shortest edge path between $v_{M_1}$ and $v_{M_2}$ as the \emph{Lickorish path length} and denote this function by $p_L(v_{M_1},v_{M_2})$.  For example, if $P$ denotes the Poincar\'e homology sphere, then $p_L(S^3, P \# P) =2$. See section \ref{sect:weightone} for more on $p_L$. Lickorish path length appears in the literature as surgery distance (see \cite{AUCKLY1}, \cite{ICHSAITO}). This gives us a metric on $\thegraph$, which we assume throughout the paper. 
 %{\color{red} nrh: 11/7 : We should probably cite Gordon-Luecke here.}
  
The Big Dehn Surgery Graph is very big.  In order to get a handle on it, we will study some useful subgraphs.  
 We denote the subgraph of a graph generated by the vertices $\lbrace v_i \rbrace$ by $\langle \lbrace v_i \rbrace \rangle$.  
 The {\it link} of a vertex $v$ is the subgraph $lk(v) = \langle w: p_L(v,w) =1 \rangle $.  If there is an automorphism 
 of $\thegraph$ taking a vertex $v$ to a vertex $w$, then the links of $v$ and $w$ are isomorphic as 
 graphs.  We study the links of vertices and a possible characterization of the link of $S^3$ in 
 $\S$\ref{sect:links}.  Associated to any knot $K$ in a manifold $M$ is a $\KI$, the
  complete graph on infinitely many vertices (to distinguish between knots and complete graphs we will use $\KI$ and $\Kn$ to denote complete graphs).  In the case, that we want to describe a $\KI$ subgraph of $\thegraph$, we will use $M_\infty^{K} = \langle v_{M'}: M' = M(K;r) \rangle$.  See 
 $\S$\ref{sect:links} for a full set of notation conventions used in this paper. 

    Interestingly, not every $\KI$ arises this way.  We prove this in $\S$\ref{sect:kinfty} and 
    make some further observations about these subgraphs.  In $\S$\ref{sect:hyperbolic_graph} we study 
    the subgraph $\thehypgraph$.  The vertices of the subgraph $\thehypgraph$ are closed hyperbolic 3-manifolds and 
    there is an edge between two vertices $v_M$ and $v_N$ if there is a cusped hyperbolic 
    3-manifold with two fillings homeomorphic to $M$ and $N$.  We also study the geometry 
    of $\thegraph$ and $\thehypgraph$, showing that neither is $\delta$-hyperbolic in $\S$\ref{sect:flats}.  In $\S$\ref{sect:flats} we also construct flats of arbitrarily large dimension in $\thegraph$.  An infinite family of hyperbolic 3-manifolds with weight one fundamental group which are not obtained via surgery on a knot in $S^3$ is given in $\S$\ref{sect:weightone}.  This shows that a characterization of the vertices in the link of $S^3$ remains open. 
     Bounded subgraphs whose vertices correspond to other geometries are detailed in $\S$\ref{ThurstonGeoms}.
     \section{Acknowledgements} 
  Both authors have benefitted from many conversations with colleagues.  We would particularly like to thank Margaret Doig for suggesting that we look at the manifolds in \cite{GRWAT}.  We are also grateful to  Steven Boyer, Nathan Dunfield, Marc Lackenby, Tao Li, Luisa Paoluzzi, and Richard Webb.  The first author was partially supported by ARC Discovery Grant DP130103694 and the Max Planck Institute for Mathematics. The second author was partially supported through NSF Grant 1207644.

\section{The link of $S^3$}\label{sect:links}
We now set notation which we will use for the remainder of the paper.  A {\it slope} on the boundary of a 3-manifold $M$ is an isotopy class of unoriented, simple closed curves on $\partial M$. We denote the result of Dehn surgery on $M$ along a knot $K \subset M$ with filling slope $r$ by $M(K;r)$.    We denote Dehn filling along a link $L=K_1 \cup K_2... \cup K_n \subset M$ by $M(\lbrace K_1,...,K_n \rbrace;(r_1,...,r_n))$ or $M(L; (r_1,...,r_n))$, with a dash denoting an unfilled component. Thus the exterior of $K$ in $M$ is denoted $M(K;-)$ and the complement is denoted by $M\setminus K$. We will say that $M(K;-)$ or $M \setminus K$ is hyperbolic if its interior admits a complete hyperbolic metric of finite volume.   For knot and link exteriors in $S^3$ we will frame the boundary tori homologically, unless otherwise noted. 

Here we study the links of vertices in \thegraph, particularly the link of $S^3$.  As above, the link of a vertex in $\thegraph$ is the subgraph $lk(v) = \langle w: p_L(v,w) =1\rangle$.  If $v$ is associated to the manifold $M$, the vertices in this subgraph correspond to distinct manifolds which can be obtained via Dehn surgery on knots in $M$.    We refer to this subgraph as {\it the link of $M$ in $\thegraph$}, or just the link of $M$.

The link of $S^3$ in $\thegraph$ is connected.  There are several proofs of this fact.  Perhaps the most intuitive is to use that a crossing change on a knot in $S^3$ can be realized as a Dehn surgery along an unknotted circle, see   \cite{CROSSING}.   One must be careful to ensure that none of these surgeries results in $S^3$. 
%Such a surgery gives an edge in the link connecting a surgery on one knot in $S^3$ with surgery on a knot with a crossing change.  Since every  knot can be transformed to the unknot, this connects every vertex in the link of $S^3$ to a vertex representing a surgery along the unknot. One must take care that this path stays strictly in the link of $S^3$, that is, that the last vertex is not $S^3$. This can be done, for example, by choosing the crossing circles carefully, since the linking numbers will determine the surgery coefficients.  Since there are knots with arbitrarily high crossing number, these paths are arbitrarily long.  See \cite{Ghys} for more on the crossing number.  

The proof we give here arose from conversations with Luisa Pauoluzzi, and the path shows that the link of $S^3$ has bounded diameter. 

\begin{proposition} The link of $S^3$ in $\thegraph$ is connected and of bounded diameter. \label{prop:connect} 
\end{proposition} 
\begin{proof}  We show any surgery on a knot in $S^3$, $S^3(K;r)$, is at most distance three in the link from a lens space.  Let $CK$ denote a  cable of $K$.  Then there is a surgery slope $pq$ and a lens space $L(p,q)$ such that $S^3(CK;pq) = S^3(K;p/q) \# L(p,q)$, see \cite{GSAT}.  Thus $S^3(CK;pq)$ is distance one from a surgery along $K$ and distance one from a lens space. \end{proof}

One might hope to distinguish the links of vertices combinatorially in $\thegraph$.  For example, is the link of any vertex in \thegraph \  connected? of bounded diameter? 
A negative answer would lead to an obstruction to automorphisms of the graph that do not fix $S^3$.  More generally, an answer to the following question would lead to a better understanding of how the Dehn surgery structure of a manifold relates to the homeomorphism type. 

\begin{question} \label{q:aut} Does the graph $\thegraph$ admit a non-trivial automorphism?  
\end{question}  

 Given our results below in Section \ref{sect:weightone}, we do not know of a conjectured answer to the following problem, which amounts to characterizing manifolds obtained via surgery on a knot in $S^3$.

 \begin{problem} Characterize the vertices in the link of $S^3$.   \end{problem}

\section{Hyperbolic examples with weight one fundamental groups}\label{sect:weightone}
A group is {\it weight $n$} if it can be normally generated by $n$ elements and no normal generating set with fewer elements exists. Recall that all knot groups are weight one and hence all manifolds obtained by surgery along a knot in $S^3$ have weight one fundamental groups. It is a folklore question if a manifold which admits a geometric structure and has a weight one fundamental group can always be realized as surgery along a knot in $S^3$ (see \cite[Question 9.23]{AFW}).    The restriction to geometric manifolds is necessary since the fundamental group of $P^3 \# P^3$ is weight one, where $P^3$ is the Poincar\'e homology sphere.  This cannot be surgery along a knot in $S^3$ since if a reducible manifold is surgery along a non-trivial knot in $S^3$, one of the factors is a lens space \cite{GORDONLUECKE}.

In Theorem \ref{thm:OurNewExamples} we show that there are infinitely many hyperbolic 3-manifolds whose fundamental groups are normally generated by one element but which are not in the link of $S^3$ [Theorem \ref{thm:OurNewExamples}]. Our technique is a generalization to the hyperbolic setting of a method of Margaret Doig, who in \cite{DOIG} first came up with examples that could not be obtained via surgery on a knot in $S^3$ using the $d-$ invariant. Boyer and Lines \cite{BOYERLINES} exhibited a different set of small Seifert fibered spaces which are weight one but not surgery along a knot in $S^3$.

\begin{figure}
\centering
% this is our normal figure call
\subfigure[]{\resizebox{2.8 cm}{!}{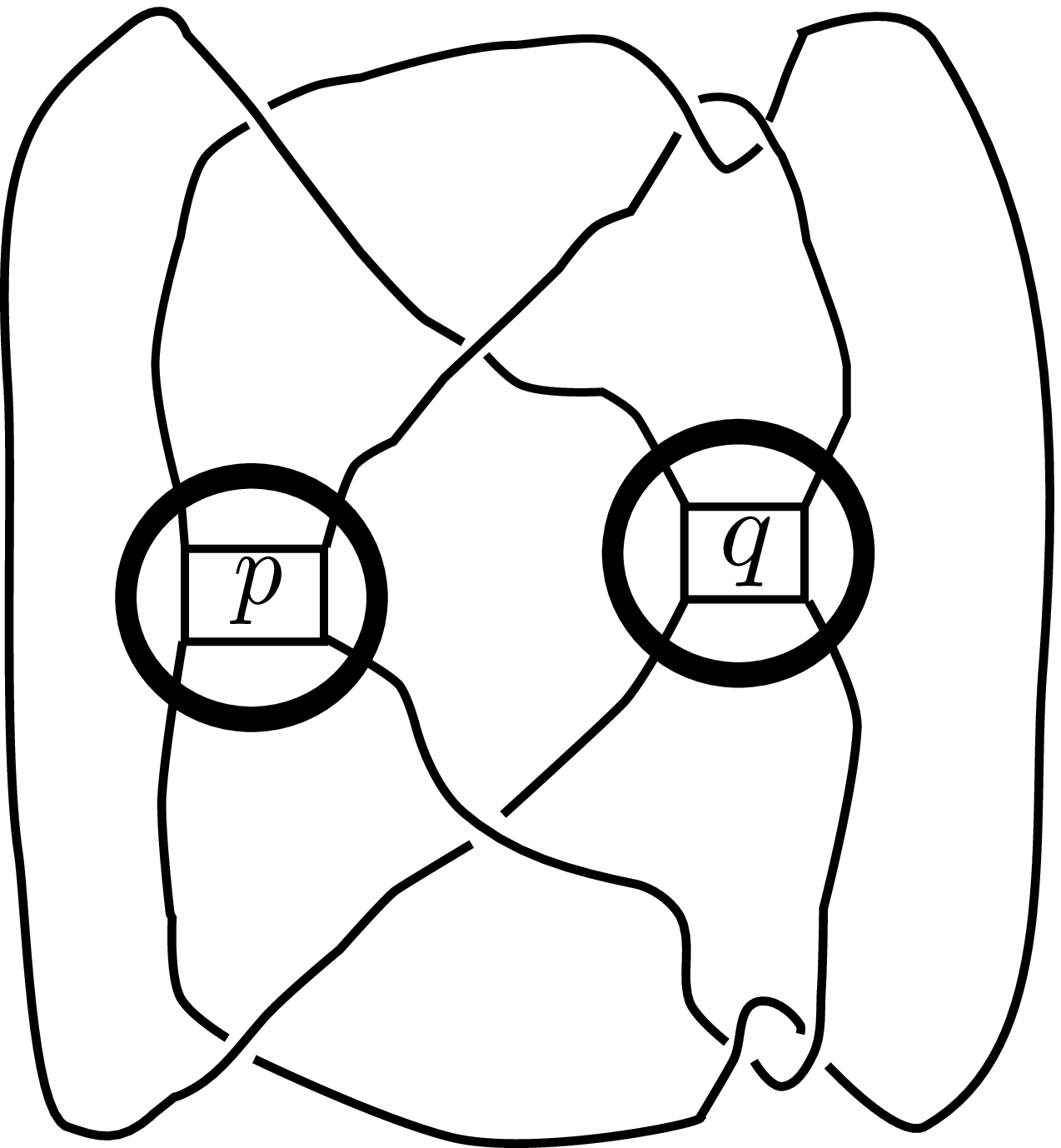}
\label{fig:KPQ}}
% this is the figure call in for the preprint
%\resizebox{6 cm}{!}{\includegraphics{../Figures/KnFilleda.pdf}}
%\caption{\label{fig:KPQ} The Kanenobu knots $K_{p,q}$}
%
%normal
\subfigure[]{\resizebox{2.8 cm}{!}{\includegraphics{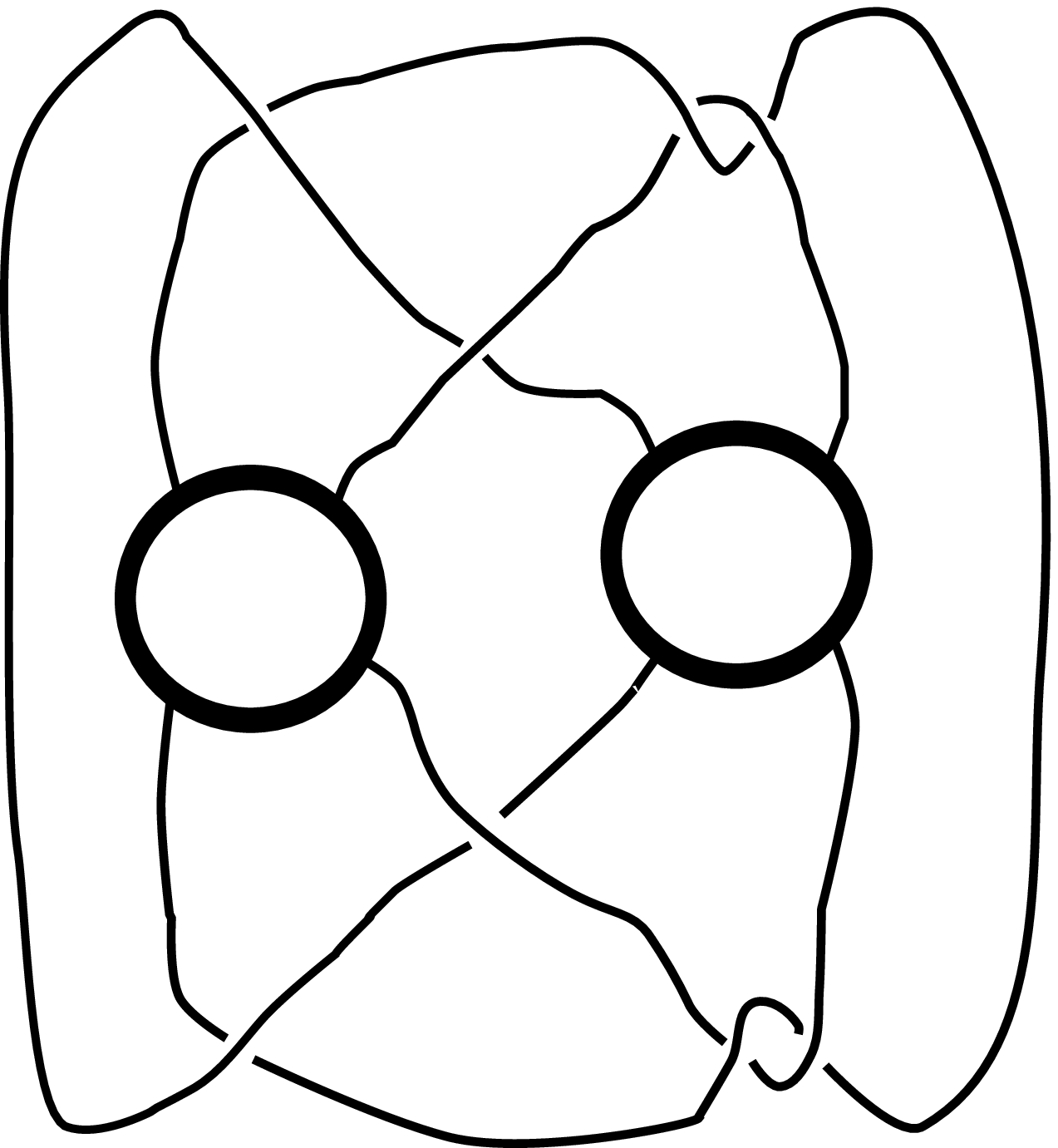}}
\label{fig:KnTwiceDrilledForHyperbolic}} 
% preprint
%\resizebox{6 cm}{!}{\includegraphics{../Figures/KnTwiceDrilledForHyperbolic}} 
%\caption{\label{fig:KnTwiceDrilledForHyperbolic} The tangle obtained from drilling the $p$ and $q$ crossing regions from $K_{p,q}$}
\caption{The The Kanenobu knots $K_{p,q}$ (left) and the tangle obtain by drilling the $p$ and $q$ twist regions (right). }
\end{figure}

Before describing the hyperbolic examples, We make a few remarks regarding the weight one condition. We have the following obstruction to surgery due to James Howie:

\begin{theorem}\cite[Corollary 4.2]{HOWIE} Every one relator product of three cyclic groups is non-trivial. 
\end{theorem}

This implies, for example,  that $M\cong L(p_1,q_1)\#L(p_2,q_2)\#L(p_3,q_3)$,  is not obtained via surgery on a knot in $S^3$, since its fundamental group is not weight one.  However, when the $p_i$ are pairwise relatively prime, its homology is cyclic. 

The following proposition extends this consequence of Howie's result to hyperbolic manifolds.

\begin{proposition}\label{prop:WeightTwo} There are hyperbolic 3-manifolds $\{N_j\}$ with cyclic homology such that for each $j$, $\pi_1(N_j)$ is weight at least two. 
\end{proposition}

\begin{proof}
Just as above, $M\cong L(p_1,q_1)\#L(p_2,q_2)\#L(p_3,q_3)$ with all the $p_i$ pairwise relatively prime. By  \cite[Theorem 1.1]{MYERSEXCELLENT}, 
there exists a knot $K \subset M$ such that $K$ bounds an immersed disk in $M$ and $M-K$ is hyperbolic. Denote 
by $\Gamma_K=\pi_1(M(K,-))$ and $\pi_1(\partial(M(K;-)))=\langle\mu, \lambda | [\mu,\lambda]=1\rangle$, 
where $\mu$ and $\lambda$ are chosen such that $M(K;\mu)=M$, and $\lambda$ bounds 
an immersed disk when considered as a curve in $M$. 

Let $\langle\langle g \rangle \rangle_{G}$ denote the normal closure of $g$ in $G$. If $\gamma$ is a curve in $\partial(M(K;-))$ representing the isotopy class $\mu^r\lambda^s$, then  
$\pi_1(M(K;\gamma))=\Gamma_K/\langle\langle \mu^r\lambda^s \rangle \rangle_{\Gamma_K}$. Observe 
that $\Gamma_K/\langle\langle \mu^r\lambda^s, \mu \rangle \rangle_{\Gamma_K} 
=\Gamma_K/\langle\langle \mu \rangle \rangle_{\Gamma_K}
= \pi_1(M)$ as $\lambda \in \langle\langle \mu \rangle \rangle_{\Gamma_K}$ since $\lambda$ bounds 
an immersed disk in $M$.  Thus, there exists a 
surjective homomorphism $f: \pi_1(M(K;\gamma)) \rightarrow \pi_1(M)$. In particular, $\pi_1(M(K;\gamma))$  is 
weight at least two. 

 If we let $N_j = M(K;\mu\lambda^j)$ then $H_1(N_j,Z)$ is cyclic of order $p_1p_2p_3$ and by Thurston's Hyperbolic Dehn Surgery Theorem \cite[Theorem 5.8.2]{THURSTONNOTES}, $N_j$ is hyperbolic for sufficiently large $j$.  \end{proof}

%
%\begin{proof}
%Just as above, $M\cong L(p_1,q_1)\#L(p_2,q_2)\#L(p_3,q_3)$ with all the $p_i$ pairwise relatively prime. By  \cite[Theorem 1.1]{MYERSEXCELLENT}, 
%there exists a knot $K \subset M$ such that $K$ bounds an immersed disk in $M$ and $M-K$ is hyperbolic. Denote $M_E = M(K,-)$ with $\pi_1(M_E) = \Gamma_E$.  We may assume that $\pi_1(M(K,-))=\langle\mu, \lambda | [\mu,\lambda]=1\rangle$, $\mu$ filling of $M_E$ is $M$, and $\lambda$ bounds an immersed disk in $M$. 
%
%If $\gamma$ is a curve in $\partial(M-n(K))$ representing the isotopy class $\mu^r\lambda^s$, then we can denote by $M_E(\gamma)$ Dehn filling along $\gamma$. Here, $\pi_1(M_E(\gamma))=\pi_1(M_E)/\langle\langle \mu^r\lambda^s \rangle \rangle_{\Gamma_E}$. Observe that $\pi_1(M_E)/\langle\langle \mu^r\lambda^s, \mu \rangle \rangle_{\Gamma_E} =\pi_1(M_E)/\langle\langle \mu \rangle \rangle_
%{\Gamma_E}= \pi_1(M)$ as $\lambda \in \langle\langle \mu \rangle \rangle_{\Gamma_E}$ since $\lambda$ bounds an immersed disk in $M$.  Thus, there exists a 
%surjective homomorphism $f: \pi_1(M_E(\gamma)) \rightarrow \pi_1(M)$. In particular, $\pi_1(M_E(\gamma))$  is weight at least 2. 
%
% If we let $N_j = M_E(1/j)$ then $H_1(N_j,Z)$ is cyclic of order $p_1p_2p_3$ and by Thurston's Hyperbolic Dehn Surgery Theorem \cite[Theorem 5.8.2]{THURSTONNOTES}, $N_j$ is hyperbolic for sufficiently large $j$.  
%\end{proof}

Over the two papers {\cite{AUCKLY2,AUCKLY1}}, Dave Auckly exhibited hyperbolic integral homology spheres that could not be surgery along a knot in $S^3$. However, it is unknown if these examples have weight one since his construction involves a 4-dimensional cobordism  that preserves homology, but not necessarily group weight.  

Margaret Doig has recently exhibited examples of manifolds admitting a Thurston geometry, but which cannot be obtained by surgery along a knot in $S^3$.

\begin{theorem}\cite[Theorem 2(c)]{DOIG}
Of the infinite family of elliptic manifolds with $H_1(Y) = \mathbb{Z}_4$, only one (up to orientation preserving homeomorphism) can be realized as surgery on a knot in $S^3$, and that is $S_4^3(T_{2,3})$.  
\end{theorem}

Although not explicitly stated in her result, for a finite group $G$, the weight of $G$ is determined by the weight $G/G'$ (see \cite{KUTZKO}), and so the above elliptic manifolds have weight one fundamental groups.

Using similar techniques and the work of Greene and Watson in \cite{GRWAT}, we are able to exhibit hyperbolic manifolds that have weight one fundamental groups but are never surgery along a knot in $S^3$. As in Greene and Watson, our examples are the double branched covers of the knots $K_{p,q}$ (see Figure \ref{fig:KPQ}) where $p=-10n$, $q=10n+3$, and $n\geq 1$. We denote these knots by $K_n$ and their corresponding double branched covers by $M_n$. The techniques of the proof may require us to omit finitely many of these double branched covers from the statement of the theorem. We will use $\lbrace X_n \rbrace$ to denote the manifolds in this (possibly) pared down set.  

%\textbf{should include remarks about hyperbolic vs d-invariant bound.}

\begin{theorem}\label{thm:OurNewExamples} There is an infinite family of hyperbolic manifolds, $\lbrace X_n \rbrace$, none of which can be realized as surgery on a knot in $S^3$.  Furthermore, these manifolds have weight one fundamental groups.
\end{theorem} 

In the following proof, we require two standard definitions from Heegaard-Floer homology
(see \cite{OZSZ4D}, \cite{DOIG}). First, a rational homology sphere $M$ is an \emph{L-space} if the hat version of its 
Heegaard Floer homology is as simple as possible, 
namely for each Spin$^c$ structure $t$ of $M$, the hat version of $\widehat{HF}(M,t)$ 
has a single generator and no cancelation.
The \emph{d-invariant}, $d(M,t)$ is an invariant assigned to each  Spin$^c$ structure $t$ of $M$ 
is the minimal degree of any non-torsion class of $HF^+(M,t)$ coming from $HF^\infty(M,t)$. 
Crudely, the d-invariant can be thought of as a way of measuring how far from $S^3$ a manifold is. 
This mentality is motivated by the argument in the proof below. 

\begin{proof}  For this proof, we use notation from \cite{OZSZRAT}.  As noted above, Greene and Watson \cite{GRWAT} study the family of knots $\lbrace K_n \rbrace$ and their double branched covers $M_n$.  The manifolds $M_n$ have the following properties: 
 
 Each $M_n$ is an $L$-space (\cite[Proposition 11]{GRWAT}).

The d-invariant, defined in \cite{OZSZ4D} of the $M_n$, satisfies the following relation: 
\begin{equation} d(M_n, i) = 2 \tau(M_n, i) - \lambda\end{equation} 

for all $n \geq 0$ and all $i \in Spin^c(M_n)$.  Here $\tau(M_n,i)$ is the Turaev torsion and $\lambda = \lambda(M_n)$ is the Casson-Walker invariant.  That the Casson-Walker invariants are all identical follows from the work of Mulllins \cite[Theorem 7.1]{MULLINS} and that the knots are ribbon and have identical Jones polynomials \cite[Propositions 8 and 11]{GRWAT}. Furthermore,  by \cite[Proposition 14]{GRWAT},
 \begin{equation} \lim_{n \rightarrow \infty} min \lbrace \tau(M_n, i) |i \in Spin^c(M_n) \rbrace = -\infty. \end{equation}

As they observe, (1) and (2) above imply:

\begin{equation} \label{unboundedcorrection} 
\lim_{n \rightarrow \infty} min \lbrace d(M_n, i) |i \in Spin^c(M_n) \rbrace = -\infty.
% the old version was this which is not about M_n
%\lim_{n \rightarrow \infty} min \lbrace \tau(d_n, i) |i \in Spin^c(M_n) \rbrace = -\infty.
\end{equation} 
 Since the manifolds $M_n$ are $L$-spaces, we may apply: 
 
 \begin{theorem} \label{alternating}  \cite[Theorem 1.2]{OZSZ4D} If a knot $K \subset S^3$ admits an $L$-space surgery, then the non-zero coefficients of $\Delta_K(T)$ are alternating $+1$s and $-1$s. 
 \end{theorem}

 Furthermore, if a knot surgery $S^3_{p/q}(K)$ with $p/q \geq 0$ is an $L$-space, it is shown in \cite[Theorem 1.2]{OZSZRAT} that the correction terms $d(S^3_{p/q}, i)$ may be calculated as follows, for $|i|\leq p/2$,  and $c = |\lfloor i/q \rfloor|$:
 
 \begin{equation} \label{boundedcorrection} d(S^3_{p/q}(K),i) -d(S^3_{p/q}(U),i) = -2 \sum_{j=1}^\infty ja_{c+j}\end{equation} 
 
 where $a_i$ is defined in terms of  the normalized Alexander polynomial of $K$: 
 
 $$\Delta_K(T) = a_0 + \sum_{i=1}^n a_i(T^i + T^{-1}).$$
 
 %The correction terms of the lens spaces can be calculated by: 
 
 %$$d(S^3_{p/q}(U),i) = - \large( \frac{pq-(2i+1 - p -q)^2}{4pq} \large) - d(S^3_{q/r}(U),j) $$ where $r \equiv p \mod q$ and $j \equiv i \mod q$ \cite{OZSZ4D}. 
 
Again, we are using notation from \cite{OZSZRAT} and in particular $S^3_{p/q}(K) = S^3(K;p/q)$.  We note that Greene and Watson \cite{GRWAT} establish that $H_1(M_n) = \mathbb{Z}/25\mathbb{Z}$. By homology considerations, if any $M_n$ is $p/q$ surgery on a (standard positively framed) knot K in $S^3$, then $p=25$.  The L-space condition implies $\frac{25}{q} \geq 2g(K)-1$ by \cite{OZSZINT,OZSZRAT} (in particular $q>0$). We also know that such a $K$ is fibered by \cite{JUHASZ,NIFIBERED} and that $g(K)$ is the degree of the symmetrized Alexander polynomial of $K$ by \cite{OZSZGENUS}, bounding the number of  terms on the right hand side of Equation (\ref{boundedcorrection}).
  
   In addition, since $M_n$ is an L-space, if $M_n = S^3_{p/q}(K)$, the Alexander polynomials of such a $K$ have bounded coefficients by Theorem \ref{alternating}. Thus, the right hand side of Equation (\ref{boundedcorrection}) is bounded and since there are only finitely many $L(p,q)$ with $p=25$, $d(S^3_{p/q}(U),i)$ only can take on finitely many values. Therefore, $d(S^3_{p/q}(K),i)$ is bounded. However, this contradicts the limit (\ref{unboundedcorrection}), and so at most finitely many of the $M_n$ can be surgery on any knot.

Next, we establish that all but at most finitely many $\{ M_n \}$ are hyperbolic.

Indeed, the Kanenobu knots $K_n$ are all obtained by tangle filling the two boundary components of the tangle $T$ in Figure \ref{fig:KnTwiceDrilledForHyperbolic}, and so the manifolds $\{ M_n \}$ are obtained by Dehn filling the double cover of $T$, which we denote by $M$.  
A triangulation for $M$ can be obtained by inputting $T$ labeled with cone angle $\frac{\pi}{2}$ into the computer software 
Orb (an orbifold version of the original Snappea) \cite{ORB} to obtain an orbifold structure $Q$.\footnote{The file and instructions on how to use it are available on the arXiv version of this paper.}
Denote by $M$,
the double cover of $Q$ corresponding to the unique index 2 torsion free subgroup 
$\pi_1^{orb}(Q)$. This computation shows that $M$ decomposes into 8 tetrahedra. 
%is labeled with cone angle $\frac{\pi}{2}$ to obtain an approximate hyperbolic structure and take the double cover to obtain a two cusped manifold $M$. 
In fact, SnapPy's identify function \cite{SNAPPY} shows $M$ is homeomorphic to `t12060' in the 8 tetrahedral census.  Also, using Snappy,  
 a set of 8 gluing equations for $M$ are encoded by the following matrix:

$$\left( \begin{array}{rrrrrrrrrrrrrrrr|r} 
 2  & -2  & 2  & 1  & 2  & -1  & 1  & -1  & 2  & 0  & 0  & 2  & -1  & 1  & 2  & 0  & -4 \\
  0  & 0  & 0  & 0  & 0  & 0  & 1  & -1  & 0  & 0  & 0  & 0  & -1  & 1  & 0  & 0  & 0 \\
 -1  & 2  & 0  & -1  & 0  & -1  & 0  & 0  & -2  & 1  & 0  & 0  & 0  & 0  & 0  & 0  & 2 \\
  0  & -2  & -1  & 1  & -1  & 1  & 0  & -1  & 0  & 0  & 0  & 0  & 0  & -1  & 0  & 0  & 0 \\
  1  & 0  & 1  & 0  & 1  & 0  & 0  & 0  & 0  & 0  & 1  & 0  & 0  & 0  & 1  & 0  & -2 \\
  0  & 0  & 1  & 0  & 1  & 0  & 1  & 0  & 2  & 0  & 0  & 0  & 1  & 0  & 0  & 0  & -2 \\
  0  & 0  & 0  & -1  & 0  & -1  & -1  & 1  & 0  & 0  & -1  & 1  & -1  & 1  & -1  & 1  & 2 \\
  0  & 0  & 0  & 0  & 0  & 0  & -1  & 1  & 0  & 0  & 1  & -1  & 1  & -1  & -1  & 1  & 0  \end{array} \right)$$ 
  
The coding is as follows given a row $ \begin{pmatrix} a_1 & b_1 & a_2 &  b_2 & ... & a_8 & b_8 &| c \end{pmatrix}$, we produce a log equation
$a_1 \log (z_1) + b_1 \log(1-z_1)+ .. +a_8 \log (z_8) + b_8 \log(1-z_8)- c \pi \cdot i = 0$. Given such an encoding $z = (2 i, \frac{1}{5} + \frac{3i}{5}, \frac{1}{5} + \frac{3i}{5}, \frac{1}{2} + \frac{i}{2}, 1  + 2i, \frac{1}{2} + i, \frac{1}{2} + \frac{i}{2}, \frac{1}{2}  + i)$ is an exact solution and therefore M and `t12060' admit a complete hyperbolic structure. By Thurston's Hyperbolic Dehn Surgery Theorem \cite[Theorem 5.8.2]{THURSTONNOTES}, the manifolds $M_n$ limit to $M$. Thus, there are at most finitely many non-hyperbolic $M_n$.

 We have that at most finitely many of the $M_n$ are surgery on a knot and at most finitely many are non-hyperbolic.  We denote the subsequence of $M_n$ that are hyperbolic and cannot be surgery along a knot by $X_n$. 
 
 Finally, we establish that $\pi_1(M_n)$ is weight one and therefore $\pi_1(X_n)$ is weight one.  As noted in \cite[$\S$4.2]{GRWAT},
 
\begin{tabular}{cl}
$ \pi_1(M_n)=\langle a_1,a_2,a_3,a_4 | b_1,b_2,b_3,b_4\rangle$ with & $b_1=(a_1^{-1} a_2)^{10n}a_4^{-1}a_1^2$,\\
& $b_2=a_2^{-1}a_3(a_2^{-1}a_1)^{10n}a_2^{-1}$,\\ 
 & $b_3=(a_4^{-1}a_3)^{10n+3}a_3^{-1}a_2 a_3^{-2}$, and\\
 & $b_4=a_1^{-1} a_4 (a_3^{-1}a_4)^{10n+3}a_4^2$.\\
\end{tabular}\\

We claim $\pi_1(M_n)/ \langle \langle a_1 \rangle\rangle_{\pi_1(M_n)}$ is trivial. 
First, the relations  $b_1$, $b_2$ become $a_2^{10n}=a_4$
and $a_2^{10n+2}=a_3$, respectively. Also, the relations $b_3$ and $b_4$ reduce to
$a_2^{10n-1}=1$ and $a_2^{10n-6}=1$ respectively. The claim follows as $gcd(10n-1,10n-6)=1$.\end{proof}

\begin{corollary}
For all $n$, $p_L(M_n,S^3)\leq 2$ and for all but at most finitely many $n$, $p_L(M_n,S^3)=2$.  
\end{corollary}

\begin{proof}
Since we can produce the unknot by switching two crossing regions of the diagram for $K_n$ as in Figure \ref{fig:dAtMostTwo}, the Montesinos trick shows that $M_n$ can be obtained from surgery along a two component link in $S^3$. Hence, we see the upper bound $p_L(M_n,S^3)\leq 2$ and   $p_L(M_n,S^3)\geq 2$ is established for all but at most finitely many $n$ by the Theorem \ref{thm:OurNewExamples}. 
\begin{figure}
\centering
\subfigure[ We drill out two crossing regions.%to produce a tangle with two cusps
]{
%Normal
\resizebox{3.2 cm}{!}{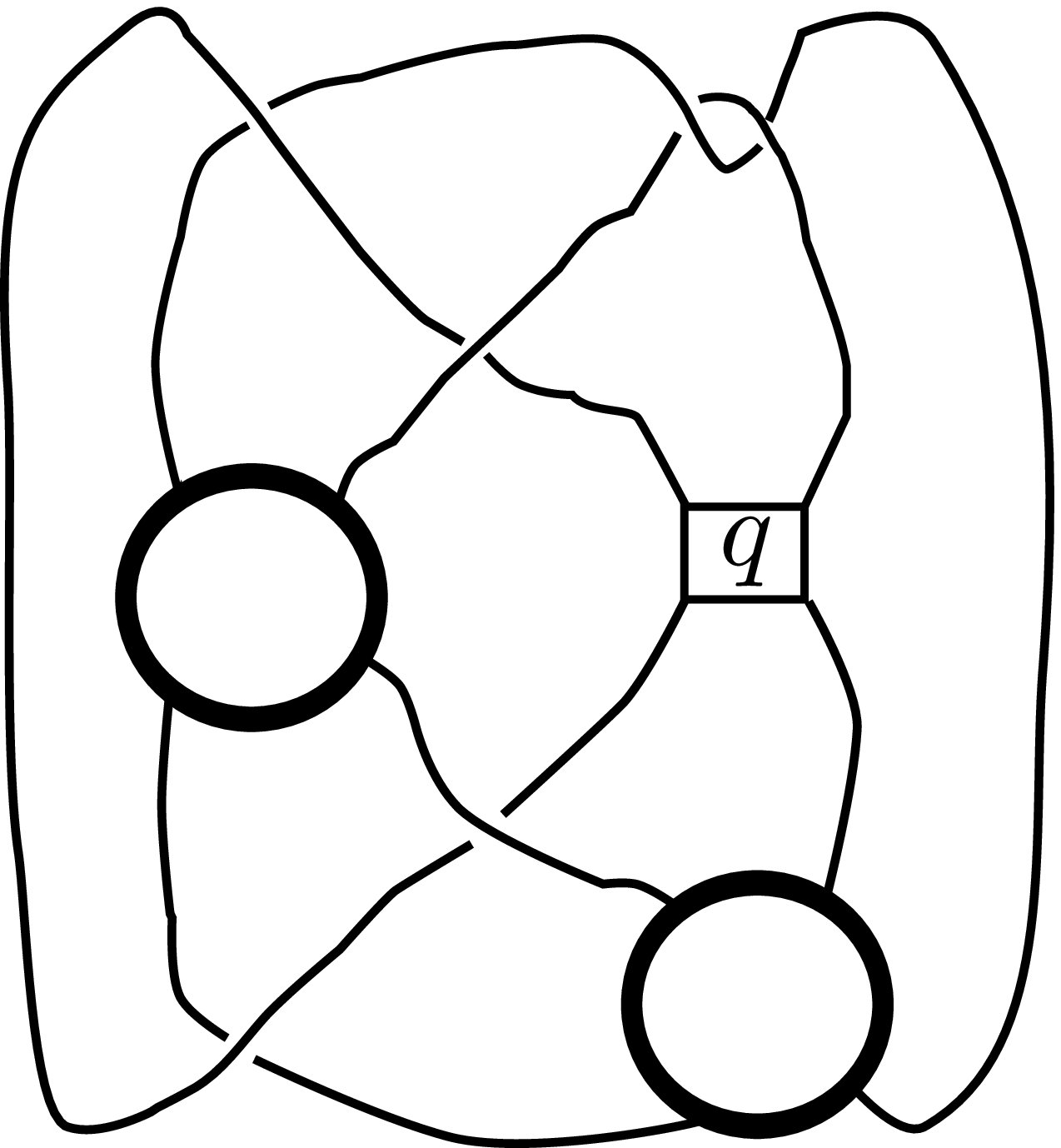}
%preprint
%\resizebox{6 cm}{!}{\includegraphics{../Figures/KnFilledTwiceDrilleda.pdf}}
\label{fig:KnFilledTwiceDrilled}
}
\hspace{2cm}
\subfigure[Then we fill as above.]{
%Normal
\resizebox{3.2 cm}{!}{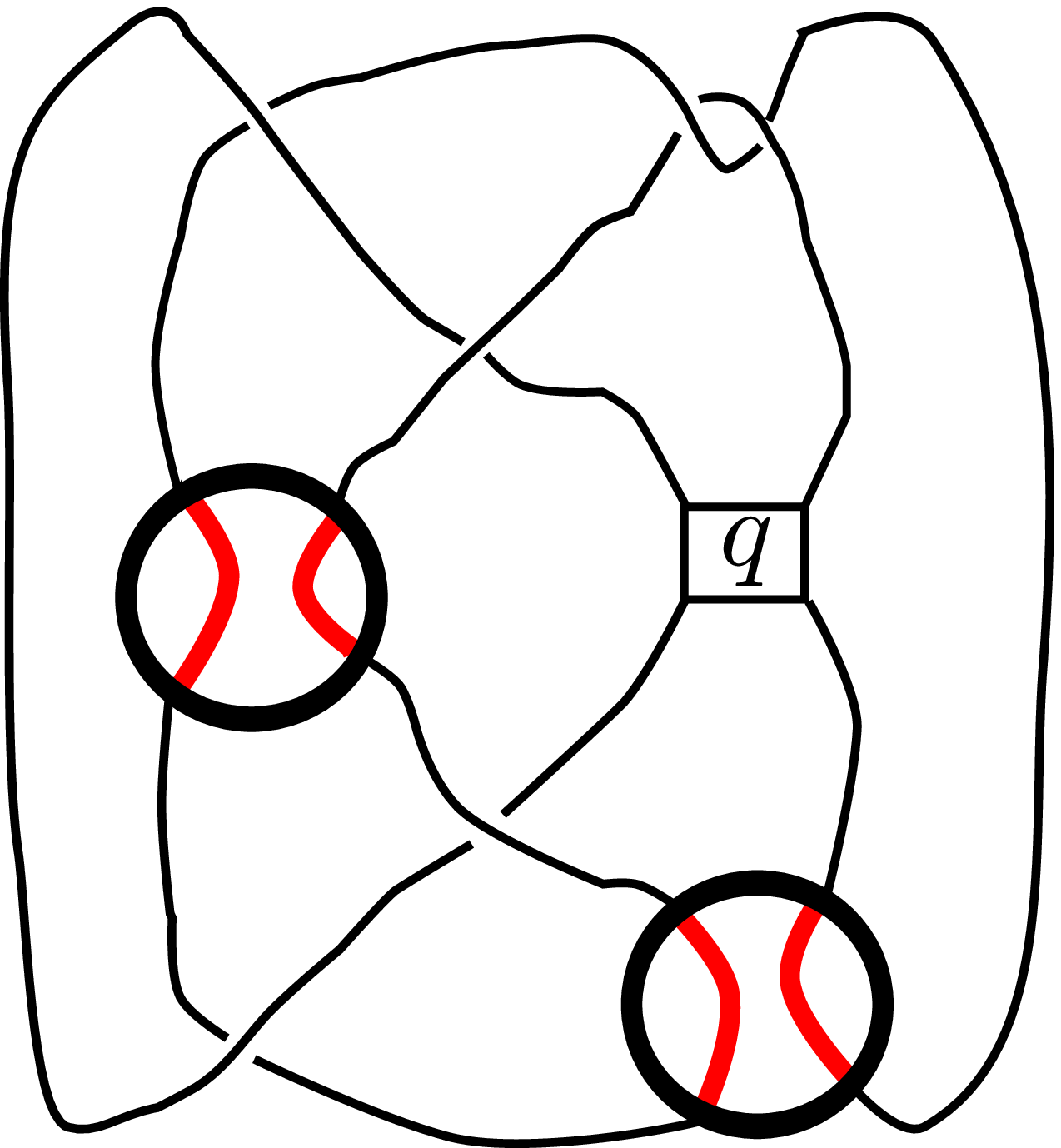} 
%preprint
%\resizebox{6 cm}{!}{\includegraphics{../Figures/KnFilledToBeUnknota.pdf}} 
\label{fig:KnFilledToBeUnknot1}
}
%\subfigure[Isotope the arc connecting the end points of the $q$ tangle region]{
%%Normal
%\resizebox{6 cm}{!}{\input{KnFilledToBeUnknotStep2.eps_tex}} 
%% preprint
%%\resizebox{6 cm}{!}{\includegraphics{../Figures/KnFilledToBeUnknotStep2a.pdf}} 
%\label{fig:KnFilledToBeUnknot2}
%}
%\subfigure[Remove the $q$ crossings to reveal a two crossing diagram of the unknot]{
%%normal
%\resizebox{!}{12cm}{\includegraphics{KnFilledToBeUnknotStep3}} 
%% preprint
%%\resizebox{6 cm}{!}{\input{KnFilledToBeUnknotStep3.eps_tex}} 
%\label{fig:KnFilledToBeUnknot3}
%}
%
\caption{\label{fig:dAtMostTwo} These diagrams show $K_n$ switching two tangle regions produces the unknot}
\end{figure}
\end{proof}

\begin{remark}
In \cite{MARENGON}, Marengon extends the techniques given here to exhibit an infinite four parameter family of double branched covers of knots given by a Kaneobu like construction. \end{remark}
%\textbf{Possibly rewrite to express length of peripheral curves, to show that all curves associated to -10n and 10n-3 surgery are long. This is delicate since the thing we are filling has 2 cusps.}

%
%\begin{figure}
%\centering
%\subfigure[ $K_{p,q}$ as a tangle
%]{
%\resizebox{6 cm}{!}{\includegraphics{KnFilleda.pdf}}
%\label{fig:KnFilled}
%}
%\subfigure[The tangle obtained from drilling the $p$ and $q$ crossing regions]{
%\resizebox{6 cm}{!}{\includegraphics{KnTwiceDrilledForHyperbolic}} 
%\label{fig:KnTwiceDrilledForHyperbolic}
%}
%\caption{\label{fig:KnHyperbolic} $K_n$ comes from tangle filling the tangle on the right}
%\end{figure}
%\end{proof}

\section{Complete infinite subgraphs}\label{sect:kinfty}
Here we discuss an interesting property which may allow one to ``see" knots in the graph $\mathcal{B}$.  
We also want to employ the notion of the set of neighbors of a vertex in a graph. More formally, for a graph $G$ and a subset $\lbrace w_i \rbrace$ of the vertices of $G$, let $\langle\lbrace w_i \rbrace \rangle$ be the subgraph induced by these vertices. That is, the vertices of $\langle\lbrace w_i \rbrace \rangle$ are $\lbrace w_i \rbrace $, and  $(w_i, w_j)$ is an edge of $\langle\lbrace w_i \rbrace \rangle$ exactly when $(w_i, w_j)$ is an edge of $G$.  Then,  as  in the introduction, we define the {\it link a vertex $v$ in $G$} to be $\langle\lbrace w_i \rbrace \rangle$, for all $w_i$ which are path length one from $v$.  

\begin{definition} If $K$ is a knot in a 3-manifold $M$ then $(M)^K _\infty = \langle \lbrace v_{M(K;r)} \rbrace  \rangle$, where $\lbrace M(K;r) \rbrace $ is the set of 3-manifolds obtained from $M$ via Dehn surgery on $K$.
\end{definition} 

\begin{proposition} For any closed 3-manifold $M$ and knot $K \subset M$, $M^K _\infty$ is a $\KI$.  
\end{proposition} 

\begin{proof}That every vertex in $M^K _\infty$ is connected to every other one is a consequence of the definition.  We just need to observe that there are infinitely many different manifolds in this subgraph. If $M \setminus K$ admits a hyperbolic structure, then all but finitely many fillings are hyperbolic.   Furthermore, the volumes approach the volume of $M \setminus K$ and so there are infinitely many different homeomorphism types. If $M \setminus K$ is Seifert-fibered (including the unknot complement in $S^3$), it is Seifert-fibered over an orbifold $O$ with boundary.  The fillings $r$ can be chosen so that the result is Seifert-fibered over an orbifold where the boundary component of $O$ is replaced with a cone point of arbitrarily high order, so the Seifert-fibered spaces are not homeomorphic.  If $M \setminus K$ admits a decomposition along incompressible tori, then, infinitely many fillings have this same decomposition \cite{TOROIDAL}. Then the boundary of $M \setminus K$ is in either a hyperbolic piece or a Seifert-fibered piece, and the above arguments apply. Finally, if $M \setminus K$ is reducible, then there exists a separating $S^2$ such that $M \setminus K = M_1 \# M_2 \setminus K$ where $M_2 \setminus K$ is irreducible. In this case, the previous arguments can be applied to yield the desired result. 
\end{proof}

Note that sometimes, the intersection of two $\KI$ subgraphs arising from fillings on knot complements may intersect in a $\KI$.  For example, let $U$ be the unknot and $T_{r,s}$ a torus knot. Let $(S^3)^U_\infty$ be the $\KI$ associated to $S^3\setminus U$, and  $(S^3)^{T_{r,s}}_\infty$ be the $\KI$ associated to  $S^3 \setminus T_{r,s}$.  Then $(S^3)^U_\infty\cap (S^3)^{T_{r,s}} _\infty$ is a $\KI$ where each vertex is a lens space (see \cite{LMOSER}). However, this phenomena cannot happen for hyperbolic manifolds. 

\begin{proposition} If $M \setminus K$ and $M' \setminus K'$ are hyperbolic and not homeomorphic, then the subgraphs $(M)^K _\infty$ and $(M')^{K'}_\infty$  have at most finitely many vertices in common.  
\end{proposition} 

\begin{proof}  Assume that $(M)^K _\infty$ and $(M')^{K'}_\infty$ have infinitely many vertices in common.  Then infinitely many of these are hyperbolic.  Denote this set by $\lbrace N_i \rbrace_i \in \mathbb{N}$.   Choose a basepoint in the thick part of each $N_i$.  Then the geometric limit of the $N_i$ is $M \setminus K$ and it is also $M' \setminus K'$, so they must be homeomorphic. See \cite{Gromovbourbaki} for background on geometric limits. \end{proof} 

%takenout
\iffalse
\begin{proposition} Let $S$ be a subgraph of $\thegraph$ such that $S$ is a $\KI$, and each vertex of $S$ corresponds to a closed hyperbolic 3-manifold.  Then if the volumes of these manifolds are uniformly bounded, infinitely many of the manifolds are obtained by fillings of a one-cusped hyperbolic 3-manifold. \end{proposition} 

I found the following theorem in Gromov's Bourbaki notes (I still need to find a real reference) 
\begin{theorem}[Jorgensen] Take a sequence of closed hyperbolic orientable 3-manifolds with uniformly bounded volumes. Then there is a subsequence $V_i$, which converges to a $V$, and a positive sequence $\epsilon_i \rightarrow 0$ such that each $V_i$ has $q$  simple closed geodesics of length $\leq \epsilon_i$ with $q$ independent of $i$.  The manifold $V$ has $q$ cusps and is diffeomorphic to each of the $V_i$, minus $q$ closed geodesics. 
\end{theorem} 
\fi 
%endtakenout

\subsection{Subgraphs which do not arise from filling} 

\begin{theorem}
 There is a $\KI$ of small Seifert fibered spaces that does not come from surgery along a one cusped manifold.
\end{theorem}

\begin{proof}
We will construct a family $M_{i,j}$, $i \in \lbrace 1,2,3,4\rbrace, j \in \mathbb{N}$ of Seifert fibered spaces over an orbifold with base space $S^2$ and negative Euler characteristic.    We follow notation in \cite{Hatcher}.  In particular, we denote a closed Seifert fibered space  by $SFS(F; \alpha_1/\beta_1, ... , \alpha_n/\beta_n)$ where $F$ is the underlying space of the base orbifold. The cone points of the base orbifold will have multiplicities $\beta_i$. The Seifert fibered invariants of the exceptional fibers are $\alpha_i/\beta_i$, which are allowed to take values in $\mathbb{Q}$.    Two Seifert fiberings $SFS(F;  \alpha_1/\beta_1, ... , \alpha_n/\beta_n)$ and $SFS(F';  \alpha_1'/\beta_1', ... , \alpha_m'/\beta_m')$ are isomorphic by a fiber preserving diffeomorphism if and only if after possibly permuting indices, $\alpha_i/\beta_i \equiv \alpha_i'/\beta_i' \mod 1$ and, if $F$ is closed, $\sum \alpha_i/\beta_i = \sum \alpha_i'/\beta_i'$. \cite[Proposition 2.1]{Hatcher}. 

Now let $\{a_1/b_1, a_2/b_2, a_3/b_3, a_4/b_4\}$ be four distinct rational numbers, such that $0 <a_i/b_i <1$,  $\sum 1/b_i <1$ and $a_i, b_i$ are relatively prime. Let $M_{4,0}$ be the Seifert fibered space over $S^2$ with three exceptional fibers labeled by $a_i/b_i$ $(i \ne 4)$. We can define $M_{1,0}$, $M_{2,0}$, $M_{3,0}$ similarly. The condition $\sum 1/b_i <1$ ensures that each manifold will be Seifert fibered over a hyperbolic orbifold. 

Note that each manifold $M_{i,0}$ has exactly two common exceptional fibers with the others mod 1, and manifolds with fibrations over hyperbolic base orbifolds have unique Seifert fibered structures \cite[Theorem 3.8]{SCOTT}. 
\begin{observation}
The set of manifolds $\{M_{i,0}\}$ form a $\Kfour$ in $\thegraph$.
\end{observation}

We will now construct a $\KI$ which consists of infinitely many of these $\Kfour$. Note that if we add $1$ to each Seifert invariant of each exceptional fiber above, we get another $\Kfour$.  Each new manifold is distinct from the manifolds in the previous $\Kfour$ since the sum of its Seifert invariants is not equal to any vertex in the original. Each vertex in the new $\Kfour$ is connected to each vertex of the previous $\Kfour$ as, for example  $SFS(S^2; a_1/b_1 +1, a_2/b_2 +1, a_3/b_3 +1) \equiv SFS(S^2; a_1/b_1+3,a_2/b_2, a_3/b_3) \equiv SFS(S^2; a_1/b_1, a_2/b_2 +3, a_3/b_3) \equiv SFS(S^2; a_1/b_1 , a_2/b_2 , a_3/b_3 +3)$.  Dehn surgery along one of the exceptional fibers can result in any manifold which is a vertex of the original $\Kfour$. Continuing this way, we have a $\KI$, parametrized by $(i,j)$, where $i \in \lbrace 1,2,3,4 \rbrace$ and $j \in \mathbb{N}$.  Specifically, $M_{i,j}$ is as follows:

$$\{M_{1,j} =  SFS(S^2; a_2/b_2 +j, a_3/b_3 +j, a_4/b_4 +j), $$
$$M_{2,j} =  SFS(S^2;  a_1/b_1 +j,  a_3/b_3 +j, a_4/b_4 +j), $$ 
$$M_{3,j} =  SFS(S^2; a_1/b_1 +j , a_2/b_2 +j, a_4/b_4 +j),  $$
$$M_{4,j} =  SFS(S^2; a_1/b_1 +j, a_2/b_2 +j, a_3/b_3 +j)\}.$$

Assume this $\KI$ comes from filling a one cusped 
manifold $M$. First, $M$ must be irreducible.  Indeed if it were reducible, there would be a two-sphere that did not bound a ball in $M$.  If the sphere is non-separating it will remain non-separating in any filling.  If it is separating, there is at most one filling of a knot in a ball which will make it a ball, \cite{GORDONLUECKE}. 

Next we observe that each $M_{i,j}$ is a small Seifert fibered space and in particular non-Haken.  
We claim that we may assume $M$ does not contain an essential torus. Indeed, if $M$ does contain an essential torus $T$, then that torus compresses in infinitely many fillings. Infinitely many fillings cannot be pairwise distance 1 and thus by \cite[Theorem 2.01]{CGLS}, $T$ and $\partial M$ cobound a cable space, $C$. 
Surgeries on cable spaces are well-understood.  As in \cite[p 179]{BOYER}, the filling of the cable space is either reducible (only along the cabling slope), a solid torus, or a manifold with an incompressible boundary torus.  Since $T$ is compressible in the fillings, $T$ bounds a solid torus in each of the filled manifolds.  Therefore, we can replace filling along $\partial M$ by filling along the torus boundary of $M \setminus C$ and get the same set of manifolds.

Thus we may assume $M$ does not contain an essential torus and since it is irreducible $M$ is geometric. 

If $M$ is not hyperbolic, it is a Seifert fibered space containing no essential  tori.   Thus $M$ must be Seifert fibered over the disk with at most two exceptional fibers. Each $M_{i,j}$ admits a unique Seifert fibration (see \cite[Theorem 3.8]{SCOTT}). Any choice of two elements $\{a_i/
b_i \}$ to label the exceptional fibers of $M$ will disagree with two exceptional fibers in one of the $M_i$, which is a contradiction to the existence of such 
an $M$.
Thus $M$ must be hyperbolic.  However, there 
are infinitely many Seifert fibered manifolds that come from surgery on $M$. This contradicts Thurston's Hyperbolic Dehn Surgery Theorem \cite[Theorem 
5.8.2]{THURSTONNOTES}. 
\end{proof}

%We can also find finite complete graphs which do not arise from surgeries along a fixed knot in a fixed 3-manifold, as pointed out by Tao Li.  For example, let $M$, $N$, and $C$ be  irreducible manifolds such that $M$ and $N$ are integral homology spheres, $C$ has finite cyclic homology and $v_N$, $v_M$ and $v_C$ form a $K_3$ in $\thegraph$. This can be obtained since, for example,  $S^3$, the Poincar\'e homology sphere, and a lens space are all surgeries on the trefoil knot.  Then the vertices $v_{M \# N}$, $v_{M \# C}$ and $v_{C \# N}$ also form a $K_3$ subgraph of $\thegraph$.  The claim is that there is no $P$ and $K$ such that all three associated manifolds are obtained by surgery on $P$ along a knot $K$.  Indeed, such a $P \setminus N(K)$ would have to be irreducible, since the components of the sphere decomposition of the manifolds in question are different.  Then one can apply \cite{GLRED} to conclude that the filling slopes must all be distance one from each other where the distance between two slopes in $\partial(P\setminus N(K))$ is their algebraic intersection number.  If we denote the filling slope for $M \# N$ by $(1,0)$, then the filling slope of $M \# C$ and of $N \# C$ must be $(n,1)$ and $(n+1,1)$ respectively.  This is a contradiction to the fact that $M \# C$ and $N \# C$ have the same finite cyclic homology. 

%The examples above of pathological behavior involve manifolds which are not hyperbolic.  This leads us to consider the Hyperbolic Big Dehn Surgery Graph, $\thehypgraph$. 

\section{Seifert Fibered Spaces and Solv Manifolds} \label{ThurstonGeoms} 
\label{sect:ThurstonGeoms}
%In the previous section, we described the link of $S^3$. We now put this question in a more general setting. Lickorish's and Wallace's proofs are distinct in the sense that Lickorish's theorem more directly exploits the $3-$manifold topology while Wallace exhibits this fact as a part of a more general treatment. Also, as part of his proof Lickorish uses a Heegaard splitting of the manifold to provide a concrete family of links with prescribed instructions for how to realize any manifold as surgery along such a link. 

Before we discuss the hyperbolic part of the graph, we briefly discuss other geometries and Seifert fibered spaces.  By compiling work of Montesinos and Dunbar, we can obtain upper bounds for any non-hyperbolic geometric manifold.  The general idea is that ``simpler" manifolds lie close to $S^3$. We begin with a theorem of Montesinos \cite{MontesinosBook}. $\chi(F)$ is the Euler characteristic of $F$.
%which relies on the standard definition of surface genus namely for a closed orientable surface of Euler characteristic $2-2g$ the genus of $F$ or $g(F)$ is $g$. For a closed non-orientable surface of Euler characteristic $2-g$ the genus of $F$, $g(F)$ is $g$.

\begin{theorem}[Montesinos] 
Let $M$ be a closed, orientable Seifert fibered space over the surface $F$ with $n$ exceptional fibers. 
\begin{enumerate}
%\item If $F$ is orientable, then $p_L(M, S^3) \leq 2g(F) + n +1$. 
%\item If $F$ is non-orientable, then $p_L(M, S^3) \leq 2g(F) + n +1$. 
\item If $F$ is orientable, then $p_L(M, S^3) \leq 3-\chi(F) + n$. 
\item If $F$ is non-orientable, then $p_L(M, S^3) \leq 5-2\chi(F) + n$. 
\end{enumerate}
\end{theorem}

This theorem follows from the discussion in \cite[Chapter 4]{MontesinosBook} (see specifically Figure 12 in that chapter). Since each link in that figure has a component labeled by $\frac{1}{b}$, mild Kirby Calculus can be applied to the links in that figure to obtain a link with one fewer component. 
To understand the geometric non-hyperbolic manifolds, it remains to investigate solv manifolds, which are either torus bundles over $S^1$ or the union of twisted $I$ bundles over the Klein bottle (see \cite[Theorem 4.17]{SCOTT}). Work of Dunbar provides an orbifold analog to this statement, namely if $Q$ is an orientable Solv orbifold, $Q$ is either a manifold as above or an orbifold with fiber $S^2(2,2,2,2)$ over $S^1$ or the union of twisted I bundles with fiber $S^2(2,2,2,2)$ (see \cite[Propostion 1.1]{DUNBAR}). Using these two results, we can obtain the following:

\begin{proposition}
%\begin{enumerate}
%\item 
(1) 
%If $M$ is a torus bundle admitting a solv geometric structure, then $p_L(M,S^3)\leq 5$. 
 If $M$ is a solv torus bundle, then $p_L(M,S^3)\leq 5$.

%\item 
(2) If $M$ is the union of twisted $I$ bundles over the Klein bottle admitting an orientable solv structure, then  $p_L(M,S^3)\leq 3$.
%\end{enumerate}
\end{proposition}

\begin{proof}
(1) By \cite{DUNBAR}, $M$ admits a 2-fold quotient $Q$ such that the base space of $Q$ is $S^1\times S^2$ and the singular locus is a four strand braid $B$. Although that paper is careful to classify such braids, the details will not be relevant to this argument. Using the Montesinos trick, we have a sequence of tangle replacements to get from $Q$ to the trivial two strand braid in $S^1\times S^2$. The first two replacements of this sequence are shown in Figure \ref{fig:SolPL5}. The resulting link is two-bridge and therefore a single rational tangle replacement yields the unknot. The trivial two strand braid can be obtained from a single rational tangle replacement on the unknot.  Hence, $p_L(M, S^1 \times S^2) \leq 4$ and $p_L(M,S^3)\leq 5$.

(2) Let $M$ is the union of twisted $I$ bundles over the Klein bottle  admitting an orientable solv structure. Then $M$ is the 2-fold quotient of $\tilde{M}$ a solv torus bundle. Moreover, $\pi_1(\tilde{M})$ is the index 2 subgroup of $\pi_1(M)$ elements that preserve the orientation of every fiber of $M$ and we may consider $\pi_1(M)= \{\pi_1(\tilde{M}), \rho \pi_1(\tilde{M}) \}$ where $\rho$ is the composition of a translation $t$ in a fiber and a symmetry of Solv taking the form $\langle x,y,z \rangle \rightarrow \langle y,x,-z \rangle$ or $\langle -y,-x,-z \rangle$.

 Denote by $Q \cong M/\langle t \rangle$ the 2-fold quotient of $M$ by $t$.
The base space of $Q$ is $S^3$ and the singular set is isotopic to the link picture in Figure \ref{fig:SolTwistedIBundles}. The rational tangle replacements in that figure yield a two-bridge link and so the double branched cover of the resulting link is
a lens space. A lens space is path length one from $S^3$, completing the proof.
\end{proof}

Immediately from this section, we have that if $M$ is a closed orientable 3-manifold which admits a Nil, $\mathbb{E}^3$, $S^2 \times \mathbb{R}$, $S^3$ or solv geometry, then $p_L(M,S^3) \leq 5$. However, these upper bounds are not in general known to be sharp. The reader is referred to Margaret Doig's work \cite{DOIG,DOIG2} for a more comprehensive treatment of which manifolds admitting an elliptic geometric structure can be obtained from surgery along a knot in $S^3$.

\begin{figure}
\begin{center}
\subfigure[ \label{fig:SolPL5} Replacement of these rational tangles gives a null-homotopic~two bridge link in~$S^2\times S^1$.
]{
%Normal
\resizebox{4 cm}{2.8 cm}{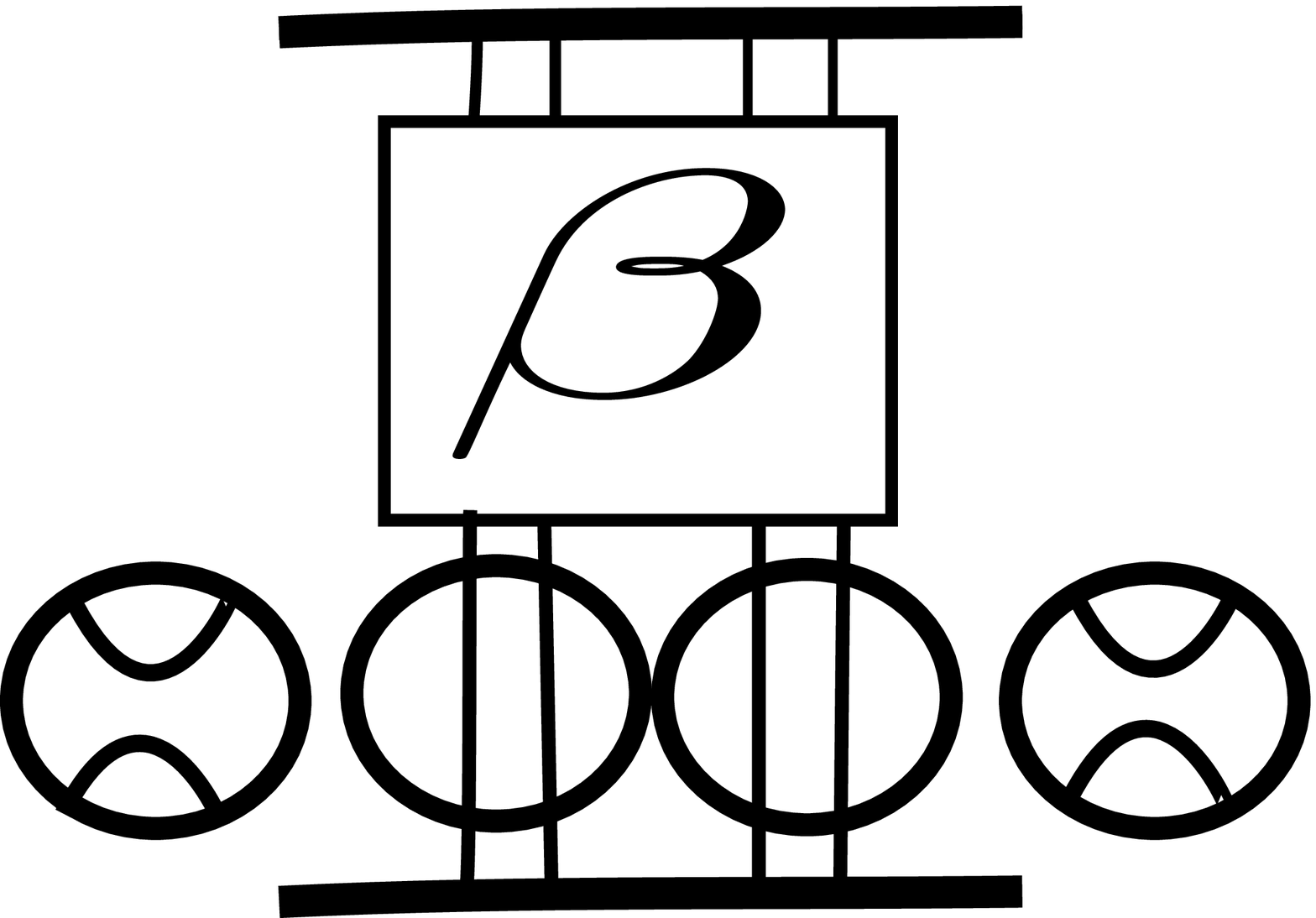}
%preprint
%\resizebox{6 cm}{!}{\includegraphics{../Figures/KnFilledTwiceDrilleda.pdf}}
}
\hspace{2 cm}
\subfigure[\label{fig:SolTwistedIBundles} Replacement of these rational tangles gives a null-homotopic two bridge link in~$S^3$.]{
%Normal
\resizebox{3.2cm}{2.8cm}{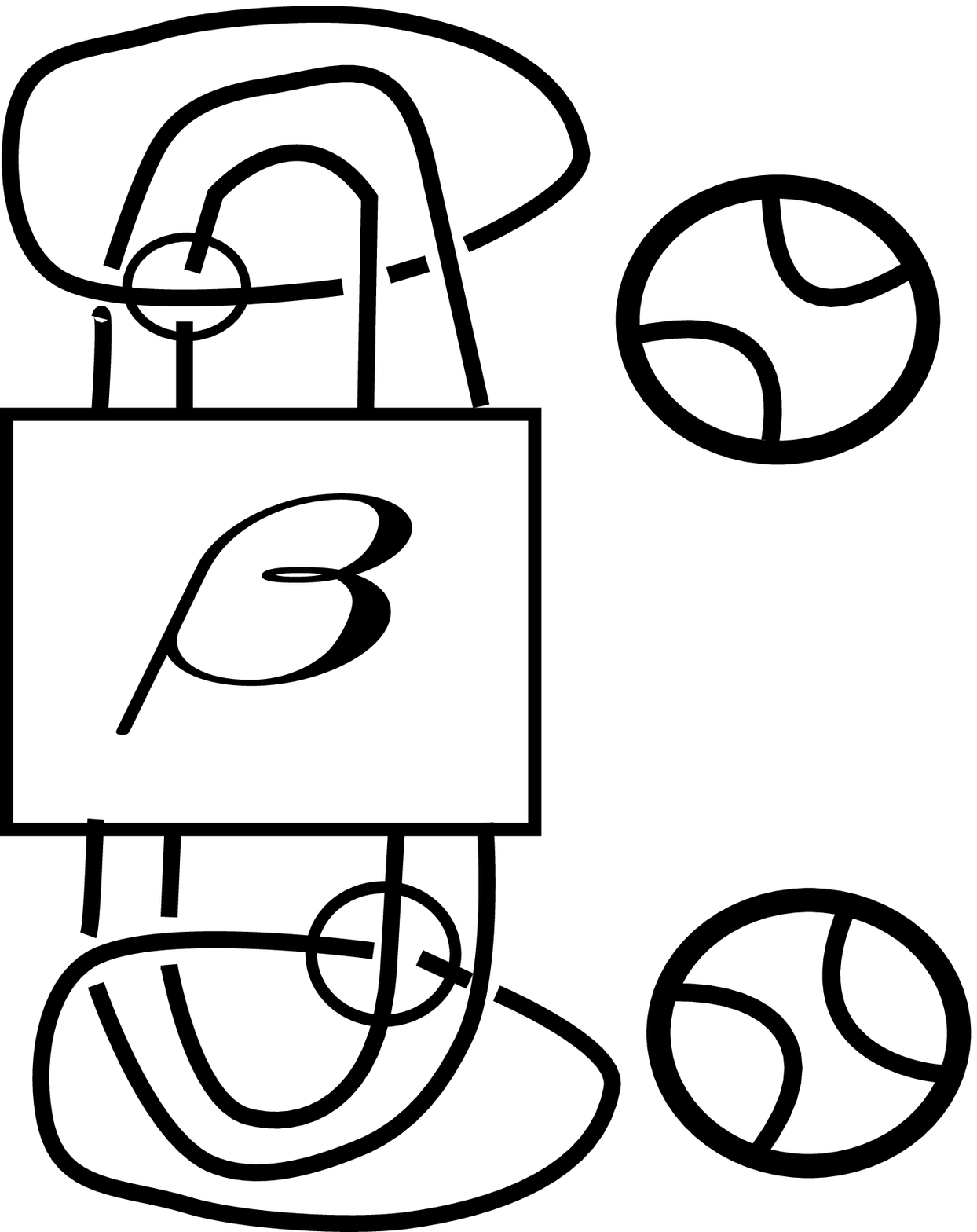}
%preprint
%\resizebox{6 cm}{!}{\includegraphics{../Figures/KnFilledToBeUnknota.pdf}} 
}
\caption{\label{fig:TorusBundles} Rational Tangle replacements on quotients of solv manifolds.}
\end{center}
\end{figure}

\section{The subgraph for hyperbolic manifolds}\label{sect:hyperbolic_graph}

\begin{definition} 
Let $\thehypgraph$ be the subgraph of $\thegraph$ such that the vertices correspond to closed hyperbolic 3-manifolds, and there is an edge between two vertices $v_M$ and $v_N$ exactly when there is a one-cusped hyperbolic 3-manifold $P$ with two fillings homeomorphic to $M$ and $N$.  
\end{definition} 

Note that there is not necessarily a hyperbolic Dehn surgery between $M$ and $N$ in our definition. For example, the surgery knot might not be represented by embedded geodesics in $M$ and $N$. 

As mentioned above in Section \ref{sect:kinfty}, this part of the graph has the nice property that if two different $K_{\infty}$  graphs that arise as $M(K)$ and $M(K')$ intersect, they must do so in finitely many vertices.  We conjecture that the combinatorics of this subgraph may reveal more of geometry and topology than in the full graph.  For the same reasons as $\thegraph$,  $\thehypgraph$ is infinite valence and infinite diameter.  We show here that it is connected, using the work of Myers.  Let $Y$ be a compact orientable 3-manifold, possibly with boundary.  Following Myers, we say that $Y$ is {\it excellent} if it is irreducible and boundary irreducible, not a 3-ball, every properly embedded incompressible surface of zero Euler characteristic is isotopic into the boundary, and it contains a two-sided properly embedded incompressible surface.  These manifolds are known by Thurston  \cite[Theorem 1.2]{THURSTONDEFORM} to admit hyperbolic structures.   By slight abuse of notation, if a properly embedded 1-manifold $K \subset Y$ has an excellent exterior, we will call $K$ excellent.

\begin{theorem}\label{thm:excellent} (Myers)  Let $M$ be a compact connected 3-manifold whose boundary does not contain 2-spheres or projective planes.  Let $J$ be a compact properly embedded 1-manifold in $M$.  Then $J$ is homotopic rel $\partial J$ to an excellent 1-manifold $K$. \end{theorem} 

For the following lemma, we observe Myers' notation in stating the following technical lemma. Namely, let $X$ be a three manifold and $F$ be a compact, (possibly disconnected) properly embedded two sided surface in $X$. Let $Y$ be the manifold obtained by cutting along $F$ and let $F_1, F_2 \subset \partial Y$ be such that identifying $F_1$ and $F_2$ yields $X$.

\begin{lemma} (Myers' Gluing Lemma, \cite[Lemma 2.1]{MYERSEXCELLENT}) \label{lem:gluing} If each component of $Y$ is excellent, $F_1 \cup F_2$ and $cl(\partial Y  \setminus ( F_1 \cup F_2) )$ are incompressible in $Y$, and each component of $F_1 \cup F_2$ has negative Euler characteristic, then $X$ is excellent. 
\end{lemma} 

\begin{theorem} 
Suppose $M_0$ and $M_{n}$ are closed hyperbolic 3-manifolds such that the associated vertices $v_{M_0}$ and $v_{M_{n}}$ are connected via a path of length $n$ in $\thegraph$.  Then $v_{M_0}$ and $v_{M_{n}}$ are connected via a path of length $n+2$ in $\thehypgraph$. \end{theorem}

\begin{proof} 
Observe that under this hypothesis, there is an n-component link, $\lbrace a_1,..., a_n \rbrace$ in $M_0$ and closed manifolds $M_1, ..., M_n$ such that 

%$$M_0 = M_0(\lbrace a_1,...,a_n\rbrace; (\alpha_1,...,\alpha_n))$$
%%$$M_1 = M_0(\lbrace a_1,...,a_n \rbrace; (\beta_1, \alpha_2,...,\alpha_n))$$
%%$$M_2 = M_0 \setminus \lbrace a_1,...,a_n\rbrace(\beta_1, \beta_2,\alpha_3,..., \alpha_n)$$ 
%$$M_i = M_0 (\lbrace a_1,...,a_n\rbrace;(\beta_1, \beta_2,...,\beta_i,\alpha_{i+1},..., \alpha_n)) \mbox{ for $1\leq i \leq n-1$}$$
%$$M_n= M_0(\lbrace a_1,...,a_n\rbrace; (\beta_1,...., \beta_n))$$

$$M_i= M_0 (\lbrace a_1,...,a_n\rbrace;(\beta_1, \beta_2,...,\beta_i,\alpha_{i+1},..., \alpha_n)), i \in \lbrace 0,..,n\rbrace$$

We will find a knot $k$ in $M_0 \setminus \lbrace a_1,...,a_n \rbrace$ and a slope $r$ such that the closed manifolds 

$$N_i= M_0 ( \lbrace a_1,...,a_n,k\rbrace;(\beta_1,...\beta_{i},\alpha_{i+1},...,\alpha_n,r)), i \in \lbrace 0,..,n\rbrace$$ are hyperbolic.  Each $N_i$ is obtained from $M_i$ via Dehn surgery on $k$ with slope $r$. The knot $k$ and slope $r$  will also have the property that the 1-cusped manifolds 

$$P_i = M_0(\lbrace a_1,...,a_n,k \rbrace;(\beta_1,...,\beta_{i-1}, -,\alpha_{i+1}...,\alpha_n,r)),$$
$$Q_0= M_0 ( \lbrace a_1,...,a_n,k \rbrace; (\alpha_1,...,\alpha_n,-)), \mbox{and } Q_n = M_0 ( \lbrace a_1,...,a_n,k \rbrace; (\beta_1,...,\beta_n,-))$$

\noindent are hyperbolic. 
We will use Myers' Theorem \ref{thm:excellent} and Lemma \ref{lem:gluing}, stated above.  We will also use the fact, proven in Lemma \ref{lem:greattangle}, that, given a $T^2 \times I$ and two slopes $x$ and $y$ on $T^2 \times \lbrace 0 \rbrace$, there is an arc $A$ in  $T^2 \times I$ with endpoints on $T^2 \times \lbrace1\rbrace$ such that the exterior $H_A$ of $A$ in $T^2 \times I$ is excellent.  Furthermore, the results of Dehn filling $H_A$ along the slopes $x$ and $y$ are excellent. 

Now we prove the existence of a knot $k$ in the exterior of the link $\lbrace a_1,...,a_n\rbrace$ in $M_0$ with the desired properties. First fix a homeomorphism $h_i$ of a neighborhood $N(\partial N(a_i))$ of each $\partial N(a_i)$ with $T^2 \times I $. For each component $a_i$, we construct an arc $A_i$ in $T^2 \times I$ such that: 
(i) $\partial A_i \subset T^2 \times \lbrace {1} \rbrace$; 
(ii) the exterior of $A_i$ in $T^2 \times I $ is excellent;
and (iii) the results of filling the exterior of $A_i$ along the slopes $h_i(\alpha_i)$ and $h_i(\beta_i)$ on $T^2 \times \lbrace 0 \rbrace$ are excellent. 
This is done in Lemma \ref{lem:greattangle} below. 
 Now by Myers' Theorem, stated as \ref{thm:excellent} above, there is an excellent collection of arcs $\lbrace B_i \rbrace$ in $M_0 \setminus \lbrace N(N(a_i)) \rbrace$ such that $B_i$  connects an endpoint of $A_i$ to one of $A_{i+1} \mod n$.  Then we claim the following: 
 
 \begin{enumerate} 
 \item $k = \cup_n (A_i \cup B_i)$ is an excellent knot in $M_0 \setminus \lbrace N(a_i) \rbrace$. 
 \item The result of filling along any choice of $\alpha_i$ or $\beta_i$ for any subset of the $a_i$ is excellent.  
 \end{enumerate} 
The fact that the union of arcs in (1) above is a knot follows from the recipe.  The fact that the exterior in (1) is excellent follows from  Myers' Lemma \ref{lem:gluing} above and the fact that each $T^2 \times I \setminus N(A_i)$ is excellent and that the exterior of the union of the $B_i$ is excellent.  Similarly,  since each $T^2 \times I \setminus N(A_i) $ filled along $\alpha_i$ or $\beta_i$ is excellent, Myers' gluing Lemma \ref{lem:gluing} yields that filling  any subset of the $a_i$ along $\alpha_i$ or $\beta_i$ is excellent.  Thus, in particular, $Q_0$ and $Q_n$ above are hyperbolic. 

Let $k$ be a knot in $M_0 \setminus \lbrace N(a_i) \rbrace$ having property (1) above. Choose a slope $r$ on $\partial N(k)$ such that $r$ lies outside of the finite set of slopes that makes any one of the closed manifolds $N_i$ or the cusped manifolds $P_i$ not hyperbolic. 

Then the path $M_0, Q_0, N_0, P_1, N_1, P_2, ...., N_{n}, Q_n, M_{n}$ is a path in $\thehypgraph$ connecting $v_{M_0}$ and $v_{M_n}$.   Here the $M_i$ and $N_i$ are closed hyperbolic manifolds (represented by vertices in $\thehypgraph$) and the $P_i$ and $Q_i$ are cusped hyperbolic manifolds (represented by edges in $\thehypgraph$). 
\end{proof}

\begin{lemma} \label{lem:greattangle} 
Given $T^2 \times [0,1]$ and two isotopy classes of curves $x$ and $y$ on $T^2 \times \lbrace 0 \rbrace$, there is an arc $A$ with endpoints on $T^2 \times \lbrace 1 \rbrace$ such that:  
\begin{enumerate} 
\item $T^2 \times I \setminus N(A)$ is excellent. 
\item The results of filling $T^2 \times I \setminus N(A)$ along the slopes $x$ and $y$ are excellent. 
\end{enumerate} 
%describe homeomorphism. 
%The there exists an $f$ such that filling $DH_E$ symmetrically along $f^{-1}(x)$ and $f^{-1}(y)$ is hyperbolic. Then The arc $A = f(E)$ in $T^2 \times I$ has the property that filling along $x$ and $y$ is hyperbolic. (indeed doubling the exterior results in f(DH_E) which is hyperbolic when Dehn filled along x) 
  
\end{lemma} 

\begin{proof}
 By Myers' Theorem \ref{thm:excellent}, there exists an arc $E$ in $T^2 \times I$ with endpoints on $T^2 \times \lbrace 1 \rbrace$  such that the exterior $T^2 \times I \setminus N(E)$ is excellent.   The arc we will use is $E$, wrapped around enough to make filling along 2 specified slopes $x$ and $y$ hyperbolic.  We detail this wrapping around below. 
 
 Fix $T^2 \times I$ up to isotopy.  Let $m$ be an oriented slope on $T^2 \times \lbrace 0 \rbrace$. Let $A_m$ be an essential annulus bounded by $m$ and a curve $m'$ on $T^2 \times \lbrace 1 \rbrace$. Let $l$ be a slope that has intersection number 1 with $m$.  There are homeomorphisms $f_m, f_l:T^2 \times I \rightarrow T^2 \times I$ obtained by cutting along $A_m$ and $A_l$, twisting once, and then gluing back by the identity on this annulus.  We twist so that an oriented $f_m(pm + ql) = pm + (q+1)l$ and $f_l(pm + ql)=  (p+1)m + ql$, in the original isotopy class of $T^2 \times I$.  Furthermore, given an $n \in \mathbb{N}$ and an oriented slope $t$, there is an $f$, which is a composition of $f_m$ and $f_l$ such that the oriented intersection of $t$ and $m$ and $t$ and $l$ is larger than $n$.  
 
Now let $H_E$ be the exterior of $E$ in $T^2 \times I$.  There is a subsurface $D=T^2 \times \lbrace 1 \rbrace \setminus N(\partial E)$ of the boundary such that it and its complement are incompressible in $H_E$.  Thus we may apply Myers' Gluing Lemma (Lemma \ref{lem:gluing}) to the double along $D$, $DH_E$ and conclude that it is excellent, hence hyperbolic.  The manifold $DH_E$ is the exterior of a knot in  $T^2 \times [0,2]$.  We say that filling along the components $T^2 \times \lbrace 0 \rbrace $ and $T^2 \times \lbrace 2 \rbrace$ such that the filling is the double along $D$ of a filling of $H_E$ is a {\it symmetric} filling.  Then, by Thurston's Hyperbolic Dehn Surgery Theorem \cite[Theorem 5.8.2]{THURSTONNOTES}, all but finitely many symmetric fillings of $DH_E$ are hyperbolic.  (Note that the filling curves have the same length in the complete structure on $DH_E$) The maps $f_m$ and $f_l$ extend naturally to $DH_E$ (by restriction to $H_E$ and doubling) and take symmetric slopes to symmetric slopes.  Thus there is a map $f: DH_E \rightarrow DH_E$, which can be taken to be of the form $f_m^nf_l^p$, such that filling $DH_E$ symmetrically along $f^{-1}(x)$ and $f^{-1}(y)$ is hyperbolic.  Then the arc $A = f(E)$ in $T^2 \times I$ has the property that filling along $x$ and $y$ is hyperbolic. Indeed doubling the exterior of $A$ results in $f(DH_E)$ which is hyperbolic when symmetrically Dehn filled along $x$ and $y$. 
\end{proof}

\section{Obstructions to $\delta-$hyperbolicity}\label{sect:flats}

We recall the following definitions. 
A geodesic metric space is $\delta$-hyperbolic if every geodesic triangle is ``$\delta$ thin", that is, every side is contained in a $\delta$-neighborhood of the union of the other 2 sides. Two metric spaces $X,Y$ with metrics $\mu_X, \mu_Y$ are \emph{quasi-isometric} 
if there exists a function $f:X \rightarrow Y$ and $A \geq1, B \geq 0, C \geq 0$ such that (1) for all $x_1,x_2 \in X$, $\frac{1}{A} \mu_Y(f(x_1,x_2)) - B \leq 
\mu_X(x_1,x_2) \leq A\cdot \mu_Y(f_1(x_1,x_2)) + B$ and (2) every point of $Y$ lies in the $C$-neighborhood of $f(X)$.  A \emph{k-quasi-flat} in a metric space $X$ is a subset of $X$ that is quasi-isometric to $\mathbb{E}^k$. In this section we will construct quasi-flats in $\thegraph$ and $\thehypgraph$, showing that these spaces are not $\delta$-hyperbolic. 

We will need to compute the exact distance in some simple examples.  To do so, we first give a method for a lower bound on the distance. 
\begin{lemma}\label{lem:AlgebraicLemma}
 Let $M_1$ and $M_2$ be closed orientable 3-manifolds and let $0\leq m \leq n$ and $p$ a prime. 
 If $\pi_1(M_1) \twoheadrightarrow \left( \Z/p\Z \right)^n$ and 
 $\pi_1(M_2) \twoheadrightarrow \left( \Z/p\Z \right)^m$ but 
 $\pi_1(M_2) \not\twoheadrightarrow \left( \Z/p\Z \right)^{m+1}$, 
 then $$p_L(M_1,M_2) \geq n-m.$$
 \end{lemma}
 
 \begin{proof}
Let $K$ be  a knot in a closed manifold $M$, and let $w$ be a word in $\pi_1(M( K;-))$. We claim that if $\phi: \pi_1(M(K;-)) \rightarrow \left( \Z/p\Z \right)^n$ is a surjection, then $\phi$ induces a surjection $\phi'$ from $\pi_1(M( K;-)/\langle\langle w\rangle\rangle$  to $\left( \Z/p\Z \right)^n$ or $\left( \Z/p\Z \right)^{n-1}$. Indeed, the image of $w$ under $\phi$ is either trivial or non-trivial.  If it is trivial, then $\phi$ induces a surjection $\phi': \pi_1(M( K;-))/\langle\langle w\rangle\rangle$  to $\left( \Z/p\Z \right)^n = \left( \Z/p\Z \right)^n/\langle\langle\phi(w)\rangle\rangle$.  If $\phi(w)$ is non-trivial, then it is order $p$ in $\left( \Z/p\Z \right)^n$, since every element is order $p$.  Then there is a minimal generating set of $\left( \Z/p\Z \right)^n$ where $\phi(w)$ is a basis element. Then $\phi': \pi_1(M( K;-))/\langle\langle w\rangle\rangle \rightarrow \left( \Z/p\Z \right)^{(n-1)}  = \left( \Z/p\Z \right)^n/\langle\langle\phi(w)\rangle\rangle$ is a surjection.  This proves the claim.   

We note that if $\pi_1(M( K;-))/\langle\langle w\rangle\rangle$ 
surjects $\left( \Z/p\Z \right)^n$, then $\pi_1(M( K;-)))$ does as 
well, since there is a presentation of the two groups which differs 
only by the addition of a relation. Then the claim implies that the 
maximum $n$ such that $\pi_1(N)$ surjects $\left( \Z/p\Z \right)^n$ 
can change by at most 1 under the operation of Dehn surgery 
along a knot in $N$, and the lemma follows. 
 \end{proof}

\begin{theorem}\label{thm:notDeltaHyp}
$\thegraph$ contains a 2-quasi-flat.  Hence $\thegraph$ is not $\delta$-hyperbolic. \end{theorem} 

\begin{proof}
For each $n$, let $U_n$ be the unlink in $S^3$ with $n$ components with the natural homological framing.  Then we will consider the manifolds:

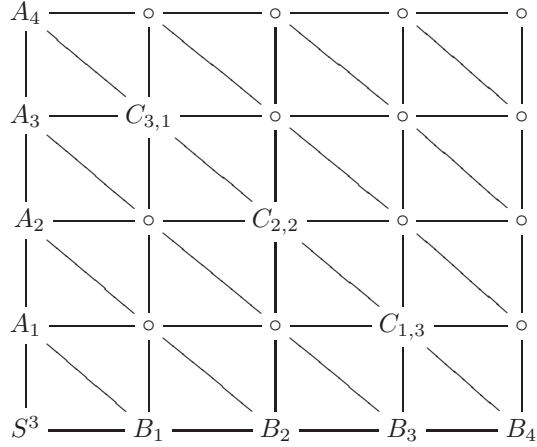
\begin{figure}
\begin{displaymath}
    \xymatrix{A_4 \ar@{-}[d] \ar@{-}[dr] \ar@{-}[r] & \circ \ar@{-}[d] \ar@{-}[d] \ar@{-}[dr] \ar@{-}[r] & \circ \ar@{-}[d] \ar@{-}[dr] \ar@{-}[r] & \circ \ar@{-}[d] \ar@{-}[r] \ar@{-}[dr] & \circ \ar@{-}[d]\\
A_3 \ar@{-}[d] \ar@{-}[dr] \ar@{-}[r] & C_{3,1} \ar@{-}[d] \ar@{-}[d] \ar@{-}[dr] \ar@{-}[r] & \circ \ar@{-}[d] \ar@{-}[dr] \ar@{-}[r] & \circ \ar@{-}[d] \ar@{-}[r] \ar@{-}[dr] & \circ \ar@{-}[d]\\\
    A_2 \ar@{-}[d] \ar@{-}[dr] \ar@{-}[r] & \circ \ar@{-}[d] \ar@{-}[d] \ar@{-}[dr] \ar@{-}[r] & C_{2,2} \ar@{-}[d] \ar@{-}[dr] \ar@{-}[r] & \circ \ar@{-}[d] \ar@{-}[r] \ar@{-}[dr] & \circ \ar@{-}[d]\\
    A_1 \ar@{-}[d] \ar@{-}[dr] \ar@{-}[r] & \circ \ar@{-}[d] \ar@{-}[d] \ar@{-}[dr] \ar@{-}[r] & \circ \ar@{-}[d] \ar@{-}[dr] \ar@{-}[r] & C_{1,3} \ar@{-}[d] \ar@{-}[r] \ar@{-}[dr] & \circ \ar@{-}[d]\\
               S^3 \ar@{-}[u] & B_1 \ar@{-}[l] &B_2 \ar@{-}[u] \ar@{-}[l] \ar@{-}[r]& B_3 \ar@{-}[r]& B_4}
\end{displaymath}
\caption{\label{fig:quasi-flat} This represents a $2$-quasi-flat in $\thegraph$. In particular, the distance between the manifolds in the figure can be determined using the edges in the figure.}
\end{figure}
%\end{small}
$$A_j = S^3(U_{2n}; (\tfrac{p_1}{1}, \tfrac{p_1}{1},...\tfrac{1}{0}, \tfrac{1}{0})),$$ 
$$B_k = S^3(U_{2n}; (\tfrac{1}{0},\tfrac{1}{0},...., \tfrac{p_1p_2}{1}, \tfrac{p_1p_2}{1})), \mbox{ and }$$ 
$$C_{j,k} = S^3(U_{2n}; (\tfrac{p_1}{1},  ....\tfrac{1}{0},\tfrac{1}{0}, \tfrac{p_1p_2}{1}, \tfrac{p_1p_2}{1}, .... \tfrac{1}{0}, \tfrac{1}{0})).$$

In other words, the surgeries on the first $n$ components are either $\tfrac{p_1}{1}$ or $\tfrac{1}{0}$,  with surgery on the first $j$ components being $\tfrac{p_1}{1}$.  Let $p_1$ and $p_2$ be distinct  primes.  The surgeries on the second $n$ components are either $\tfrac{p_1p_2}{1}$ or $\tfrac{1}{0}$, with the first $k$ being $\tfrac{p_1p_2}{1}$.  The number of non-trivial fillings of $A_j$ is $j$, of $B_k$ is $k$, and of $C_{j,k}$ is $j+k$.  Thus we can take $n$ as large as needed and these are still well-defined. 

Then: 

$$H_1(A_j, \mathbb{Z}) = (\Z/p_1 \Z)^j,$$
$$H_1(B_k, \mathbb{Z})  = (\Z/ \ p_1p_2 \Z)^k, \mbox{ and }$$ 
$$H_1(C_{j,k}, \mathbb{Z})  = (\Z / p_1 \Z)^j \oplus (\Z / p_1p_2 \Z)^k.$$ 

Then, by repeated use of Lemma \ref{lem:AlgebraicLemma}, since every map to an abelian group factors through the homology, the distances between these manifolds are as in the diagram. 
\end{proof}

Using the same methods, we can show. 
\begin{theorem}\label{cor:largeQuasi-flats}
$\thegraph$ has a $4$-quasi-flat based at $S^3$. Furthermore, $\thegraph$ has a $4$-quasi-flat based at each vertex $v_M$. 
\end{theorem}

\begin{proof}
By choosing four distinct primes $p_1,p_2,p_3$ and $p_4$, the graph $\thegraph$ can be seen to exhibit a large quasi-flat based at $S^3$. 
The vertices of such a quasi-flat are:

$$E_j = S^3(U_{2n}; (\tfrac{p_1}{1}, \tfrac{p_1}{1},...\tfrac{1}{0}, \tfrac{1}{0})), $$ 
$$N_j = (S^3(U_{2n}; (\tfrac{1}{0},\tfrac{1}{0},...., \tfrac{p_1p_2}{1}, \tfrac{p_1p_2}{1})),$$ 
$$W_j = S^3(U_{2n}; (\tfrac{p_1p_2p_3}{1}, \tfrac{p_1p_2p_3}{1},...\tfrac{1}{0}, \tfrac{1}{0})), \mbox{ and }$$ 
$$S_j = S^3(U_{2n}; (\tfrac{1}{0},\tfrac{1}{0},...., \tfrac{p_1p_2p_3p_4}{1}, \tfrac{p_1p_2p_3p_4}{1})).$$ 

In fact, if the manifolds $E_j,N_j,W_j$ and $S_j$ are connect summed with a given closed orientable manifold $M$, then by the same homology argument as above, there is a quasi-flat based at $M$. 
\end{proof}

\begin{remark}
This construction can be adapted to construct $k$-quasi-flats for arbitrarily large $k$. \end{remark}

The behavior of homology under Dehn filling is a key property of the $2n$ component split link in the argument above. The 
pairwise linking number of the components of that link is 0. The rest of this section will be 
devoted to finding an $2n+1$ component link that has similar behavior. First, we construct hyperbolic link where
 each one of the pairwise linking numbers is  $0$. This is accomplished as follows:

\begin{lemma}\label{lem:hypLinkingNoZero}
There is a knot $K$ in the complement of the $m$ component split link $U_{m}$ such that $S^3(U_{m}\cup K;-)$ is hyperbolic 
and all components have pairwise linking number $0$. 
\end{lemma}

\begin{proof}
The proof is similar to the methods in Section \ref{sect:hyperbolic_graph}. Consider the link exterior $M = S^3(U_{m};-)$ and label each boundary component by $C_i$. 
Let $H_i$ be a neighborhood of each $C_i$ in $M$ By \cite{MYERSEXCELLENT}, we can drill out a set of excellent arcs $a_i$ $1\leq i < m$ from $N= M \setminus  \bigcup H_i$ such that $a_i$ 
connects a neighborhood of the $i$th component with the $(i+1)$st component and $a_m$ connects 
the last component to the first. Furthermore, orient each $a_i$ such 
that $a_i$ is based on a neighborhood of the $i$th component.
For convenience denote the last arc by $a_0$ and $a_m$. 
If $D_i$ is the disk in $M$ with $C_i$ as a boundary, the union of the the arcs in $N$ will 
have oriented intersection number $\lambda_i$ with $D_i \cap N$. 

Let $A_i = H_i \cap D_i$. With a slight abuse of notation we consider $H_i$ as homeomorphic to $T^2 \times I$ with
the marked annulus $A_i$ embedded in it. Again by \cite{MYERSEXCELLENT}, we can 
drill out an excellent arc from $H_i$ in any homotopy class, and hence with any intersection number with $A_i$, connecting the endpoints of $a_i$ and $a_{i-1}$. Here, we choose $-\lambda_i$ to be this intersection number. 

Let $K = \cup a_i$ be an oriented knot in $S^3(L_m;-)$.
For each component $C_i$ of $L_{m}$, the disk $D_i$ is also a Seifert surface for $C_i$. Thus, the pairwise linking number of $C_i$ and 
$K$ is the oriented intersection number of $K$ with $D_i$, $0= \lambda_i-\lambda_i$.   
\end{proof}

In the above proof, there is a special component of the link, $K$, such that drilling out $K$ from $S^3(L_m;-)$ is hyperbolic. We call this component of the link $L_m \cup K$ the \emph{Myers component}. For a general $n$-component link an $(n-1)$-component must be specified to determine the Myers component.

%
%If for every abelian quotient $Q$ of $H_1(M \setminus L)$, there exists a Dehn filling of $M$ such that $H_1((M,L,\alpha)) \cong Q$ we say that the homology of  $M-L$ is {\it detected} by Dehn filling. We now show that the link $L'=L_{n} \cup K$ has homology detected by Dehn filling.

\begin{theorem}\label{thm:BHnotDeltaHyp}
$\thehypgraph$ is not $\delta$-hyperbolic.
\end{theorem}

\begin{proof} As above, we construct a quasi-flat. 
Using the link $L'=L_{n} \cup K$ as in Lemma
 \ref{lem:hypLinkingNoZero}, we have that $S^3(L';-)$ is 
 hyperbolic and each pair of components has linking number $0.$  
 This condition implies that $K$, an embedded curve, is null homologous in 
 $$S^3(L';(\tfrac{r_1}{s_1},...,\tfrac{r_n}{s_n}, \tfrac{1}{0}))\cong L(r_1,s_1) \#... \#L(r_n,s_n),$$
 since the homology class of $K$ is determined by the sum of the oriented mod $r_i$ intersection number
 with the Seifert surface of the $i$th component of $L_n$. One can be observe this directly by consideration of $K$ as curve in $S^3(L';(\tfrac{1}{0},...,\tfrac{1}{0}, \tfrac{r_i}{s_i},\tfrac{1}{0},...,\tfrac{1}{0}, \tfrac{1}{0}))\cong L(r_i,s_i)$.

 Thus, $$H_1(S^3(L'; (\tfrac{r_1}{s_1},...,\tfrac{r_n}{s_n}, \tfrac{1}{q'})),\Z)=  
 \Z/r_1\Z \oplus ... \oplus \Z/r_n \Z,$$ 
\noindent{} and so we can choose surgery coefficients such that the homology of the fillings behaves
analogously to the manifolds $A_j$,$B_k$, and $C_{j,k}$ as in the proof of Theorem \ref{thm:notDeltaHyp}. 
% The proof of the claim follows from ....
% 
% 
% 
% 
% 
% However, this is indeed the case as
%  $\tfrac{1}{1}$ surgery on any of the non-Myers component of $S^3
% (\{L'\};-)$ results in a ${n}$ component link in $S^3$ such that the 
% pairwise linking number 0 condition is preserved. Thus, 
% $\tfrac{p}{q}$ surgery contributes a factor of $\Z/p\Z$ to the 
% homology of the filling and the Meyers component remains null 
% homologous as a curve in $S^3(\{L'\}; (\tfrac{p_1}{q_1},....\tfrac
% {p_n}{q_n},\tfrac{1}{q'}))$.  
 
 Finally, we remark that choosing 
 sufficiently large choices of primes $p_1$ and $p_2$ and a large 
 choice of $q'$, the manifolds obtained by filling the first $n$ 
 components of $S^3(L'; (-,...,-, \tfrac{1}{q'}))$ by either $\tfrac
 {p_1}{1}$ or $\tfrac{p_1 p_2}{1}$ is hyperbolic by %in accordance with 
 Thurston's Hyperbolic Dehn Surgery Theorem \cite[Theorem 5.8.2]{THURSTONNOTES}.
 \end{proof}

\bibliography{BigDehn2015March-for-arxiv}

\begin{thebibliography}{10}

\bibitem{AFW}
Matthias Aschenbrenner, Stefan Friedl, and Henry Wilton.
\newblock 3-manifold groups.
\newblock arXiv:1205.0202[math.GT], 2013.

\bibitem{AUCKLY2}
David Auckly.
\newblock An irreducible homology sphere which is not {D}ehn surgery on any
  knot.
\newblock preprint available at \url{http://www.math.ksu.edu/~dav/}.

\bibitem{AUCKLY1}
David Auckly.
\newblock Surgery numbers of {$3$}-manifolds: a hyperbolic example.
\newblock In {\em Geometric topology ({A}thens, {GA}, 1993)}, volume~2 of {\em
  AMS/IP Stud. Adv. Math.}, pages 21--34. Amer. Math. Soc., Providence, RI,
  1997.

\bibitem{BOYER}
Steven Boyer.
\newblock Dehn surgery on knots.
\newblock In {\em Handbook of geometric topology}, pages 165--218.
  North-Holland, Amsterdam, 2002.

\bibitem{BOYERLINES}
Steven Boyer and Daniel Lines.
\newblock Surgery formulae for {C}asson's invariant and extensions to homology
  lens spaces.
\newblock {\em J. Reine Angew. Math.}, 405:181--220, 1990.

\bibitem{SNAPPY}
Marc Culler and Nathan~M. Dunfield.
\newblock Snappy, a computer program for studying the geometry and topology of
  3-manifolds.
\newblock available at \url{http://snappy.computop.org}.

\bibitem{CGLS}
Marc Culler, C.~McA. Gordon, J.~Luecke, and Peter~B. Shalen.
\newblock Dehn surgery on knots.
\newblock {\em Ann. of Math. (2)}, 125(2):237--300, 1987.

\bibitem{DOIG}
Margaret~I. Doig.
\newblock Finite knot surgeries and heegaard floer homology.
\newblock arXiv:1201.4187[math.GT], 2012.

\bibitem{DOIG2}
Margaret~I. Doig.
\newblock Obstructing finite surgery.
\newblock arXiv:1302.6130[math.GT], 2013.

\bibitem{DUNBAR}
William~D. Dunbar.
\newblock Classification of solvorbifolds in dimension three. {I}.
\newblock In {\em Braids ({S}anta {C}ruz, {CA}, 1986)}, volume~78 of {\em
  Contemp. Math.}, pages 207--216. Amer. Math. Soc., Providence, RI, 1988.

\bibitem{GSAT}
C.~McA. Gordon.
\newblock Dehn surgery and satellite knots.
\newblock {\em Trans. Amer. Math. Soc.}, 275(2):687--708, 1983.

\bibitem{GORDONLUECKE}
C.~McA. Gordon and J.~Luecke.
\newblock Knots are determined by their complements.
\newblock {\em Bull. Amer. Math. Soc. (N.S.)}, 20(1):83--87, 1989.

\bibitem{TOROIDAL}
C.~McA. Gordon and J.~Luecke.
\newblock Toroidal and boundary-reducing {D}ehn fillings.
\newblock {\em Topology Appl.}, 93(1):77--90, 1999.

\bibitem{GRWAT}
Joshua~Evan Greene and Liam Watson.
\newblock Turaev torsion, definite 4-manifolds, and quasi-alternating knots.
\newblock {\em Bulletin of the London Mathematical Society}, 45(5):962--972,
  2013.

\bibitem{Gromovbourbaki}
Michael Gromov.
\newblock Hyperbolic manifolds (according to {T}hurston and {J}\o rgensen).
\newblock In {\em Bourbaki {S}eminar, {V}ol. 1979/80}, volume 842 of {\em
  Lecture Notes in Math.}, pages 40--53. Springer, Berlin, 1981.

\bibitem{Hatcher}
Alan~E. Hatcher.
\newblock Notes on basic 3-manifold topology.
\newblock available at
  \url{http://www.math.cornell.edu/~hatcher/3M/3Mdownloads.html}, 2000.

\bibitem{ORB}
Damian Heard.
\newblock Orb.
\newblock available at \url{www.ms.unimelb.edu.au/~snap/orb.html}, 2005.

\bibitem{HOWIE}
James Howie.
\newblock A proof of the {S}cott-{W}iegold conjecture on free products of
  cyclic groups.
\newblock {\em J. Pure Appl. Algebra}, 173(2):167--176, 2002.

\bibitem{ICHSAITO}
Kazuhiro Ichihara and Toshio Saito.
\newblock Surgical distance between lens spaces.
\newblock {\em Tokyo J. of Math.}, 34(1):153--164, 2011.

\bibitem{JUHASZ}
Andr{\'a}s Juh{\'a}sz.
\newblock Holomorphic discs and sutured manifolds.
\newblock {\em Algebr. Geom. Topol.}, 6:1429--1457, 2006.

\bibitem{KUTZKO}
P~Kutzko.
\newblock On groups of finite weight.
\newblock {\em Proceedings of the American Mathematical Society},
  55(2):279--280, 1976.

\bibitem{LICKORISH}
W.~B.~R. Lickorish.
\newblock A representation of orientable combinatorial {$3$}-manifolds.
\newblock {\em Ann. of Math. (2)}, 76:531--540, 1962.

\bibitem{MARENGON}
Marco Marengon.
\newblock On $ d $-invariants and generalised kanenobu knots.
\newblock {\em arXiv preprint arXiv:1412.3433}, 2014.

\bibitem{MontesinosBook}
Jos{\'e}~Mar{\'{\i}}a Montesinos.
\newblock {\em Classical tessellations and three-manifolds}.
\newblock Universitext. Springer-Verlag, Berlin, 1987.

\bibitem{LMOSER}
Louise Moser.
\newblock Elementary surgery along a torus knot.
\newblock {\em Pacific J. Math.}, 38:737--745, 1971.

\bibitem{MULLINS}
David Mullins.
\newblock The generalized {C}asson invariant for {$2$}-fold branched covers of
  {$S^3$} and the {J}ones polynomial.
\newblock {\em Topology}, 32(2):419--438, 1993.

\bibitem{MYERSEXCELLENT}
Robert Myers.
\newblock Excellent {$1$}-manifolds in compact {$3$}-manifolds.
\newblock {\em Topology Appl.}, 49(2):115--127, 1993.

\bibitem{NIFIBERED}
Yi~Ni.
\newblock Link {F}loer homology detects the {T}hurston norm.
\newblock {\em Geom. Topol.}, 13(5):2991--3019, 2009.

\bibitem{OZSZ4D}
Peter Ozsv{\'a}th and Zolt{\'a}n Szab{\'o}.
\newblock Absolutely graded {F}loer homologies and intersection forms for
  four-manifolds with boundary.
\newblock {\em Adv. Math.}, 173(2):179--261, 2003.

\bibitem{OZSZGENUS}
Peter Ozsv{\'a}th and Zolt{\'a}n Szab{\'o}.
\newblock Holomorphic disks and genus bounds.
\newblock {\em Geom. Topol.}, 8:311--334, 2004.

\bibitem{OZSZINT}
Peter~S. Ozsv{\'a}th and Zolt{\'a}n Szab{\'o}.
\newblock Knot {F}loer homology and integer surgeries.
\newblock {\em Algebr. Geom. Topol.}, 8(1):101--153, 2008.

\bibitem{OZSZRAT}
Peter~S. Ozsv{\'a}th and Zolt{\'a}n Szab{\'o}.
\newblock Knot {F}loer homology and rational surgeries.
\newblock {\em Algebr. Geom. Topol.}, 11(1):1--68, 2011.

\bibitem{CROSSING}
Martin Scharlemann.
\newblock Crossing changes.
\newblock {\em Chaos, Solitons and Fractals}, 9:693--704, 1998.

\bibitem{SCOTT}
Peter Scott.
\newblock The geometries of {$3$}-manifolds.
\newblock {\em Bull. London Math. Soc.}, 15(5):401--487, 1983.

\bibitem{THURSTONNOTES}
William Thurston.
\newblock The geometry and topology of 3-manifolds.
\newblock Princeton University, Mimeographed lecture notes, 1977.

\bibitem{THURSTONDEFORM}
William~P. Thurston.
\newblock Hyperbolic structures on {$3$}-manifolds. {I}. {D}eformation of
  acylindrical manifolds.
\newblock {\em Ann. of Math. (2)}, 124(2):203--246, 1986.

\bibitem{WALLACE}
Andrew~H. Wallace.
\newblock Modifications and cobounding manifolds.
\newblock {\em Canad. J. Math.}, 12:503--528, 1960.

\end{thebibliography}
\bibliographystyle{plain}

\end{document}